\documentclass[letterpaper, 11pt]{article}
\usepackage{amsmath,amssymb, amsthm, dsfont}
\pdfoutput=1
\usepackage{graphicx}
\usepackage{multirow}
\usepackage{array}
\usepackage{sidecap}
\usepackage{fullpage}
\usepackage{comment}
\usepackage{tikz}
\theoremstyle{plain} \newtheorem{theorem}{Theorem}  \newtheorem{lemma}{Lemma} \newtheorem*{corollary}{Corollary}

\theoremstyle{definition} \newtheorem{definition}{Definition}   

\theoremstyle{remark}   

\begin{document}
\title{A Linking Number Definition of the Affine Index Polynomial and Applications}
\author{Lena C. Folwaczny and Louis H. Kauffman}
\date{}
\maketitle

\begin{abstract}
This paper gives an alternate definition of the Affine Index Polynomial (called the Wriggle Polynomial) using virtual linking numbers and explores applications of this polynomial.   In particular, it proves the Cosmetic Crossing Change Conjecture for odd virtual knots and pure virtual knots.  It also demonstrates that the polynomial can detect mutations by positive rotation and proves it cannot detect mutations by positive reflection.  Finally it exhibits a pair of mutant knots that can be distinguished by a Type 2 Vassiliev Invariant coming from the polynomial.
\end{abstract}

\section{Introduction}
This paper gives a linking number definition of the Affine Index Polynomial as described by Kauffman \cite{AIP}.  The polynomial is calculated by assigning a weight at each crossing of the diagram.  In virtual knot theory, a link of two components has \textit{two} linking numbers (as we shall explain below).  The \textit{wriggle number} of a virtual link is defined to be the difference between its two virtual linking numbers.  We denote the wriggle number of a link L as $W(L)$.  In this paper, for each crossing c in a knot diagram we smooth c in an oriented manner to obtain a link, $L_c$.  We calculate the wriggle number of this link, $W(L_c)$, and associate it to the crossing as a weight.  In the knot diagram, the weights assiciated to each crossing are then extended to a polynomial invariant, the Wriggle Polynomial, by:\\
\centerline{$W_K(t) = \displaystyle\sum_{c \in C}sign(c) t^{W(L_c)} - writhe(K)$}
 $L_c$ is the link obtained after an oriented smoothing of crossing c, $W(L_c)$ is the wriggle number of $L_c$,  and sign(c) is the sign of crossing c.  We prove this linking number definition is equal to the original definition of the Affine Index Polynomial.  Our new linking number definition gives us additional insight into the invariant and an alternative method of calculation.  The polynomial is then used to prove the Crossing Change Conjecture (see section 4) for certain classes of virtual knots.  We also demonstrate that the polynomial is successful in detecting some classes of mutant virtual knots, and prove it cannot detect others.  Certain proofs are made clearer when using a particular definition of the polynomial, so different reformulations of the definition are used throughout the proofs in the paper.\\

The paper is organized as follows.  In section 2 we review Virtual Knot Theory and other background needed for the paper.  In section 3 we introduce the Wriggle Polynomial and prove its invariance under Reidemeister moves.  In section 4 we review the Crossing Change Conjecture and give a solution for certain classes of virtual knots.  We remark that section 3 and section 4 work exclusively with the linking number definition of the Wriggle Polynomial.  In section 5 we review the definition of the Affine Index Polynomial and show the Wriggle Polynomial is equal to the Affine Index Polynomial.  After section 5, we allow ourselves to use both definitions in calculations and proofs.  In section 6 we show the polynomial can distinguish families of mutant knots of a certain type, and prove it cannot detect mutations of another type.\\

\section{Background}
\subsection{Virtual Knot Theory}
Recall that a virtual knot is a 4-valent planar graph endowed with extra information at each crossing.  We can describe each crossing as positive, negative, or virtual, as shown in Figure 1.  We can also describe knots via Gauss Diagrams, as shown in Figure 2.\\

\begin{figure}[h]
\begin{center}
\includegraphics[scale = 0.7]{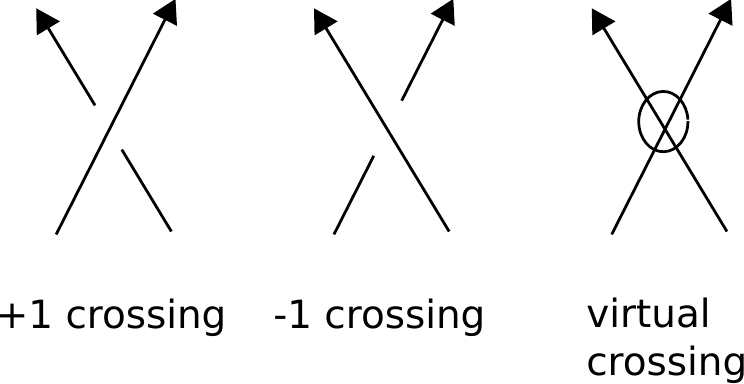}
\end{center}
\caption{The 3 types of crossings in a virtual knot diagram}
\end{figure}

\begin{figure}[h]
\begin{center}
\includegraphics[scale=0.6]{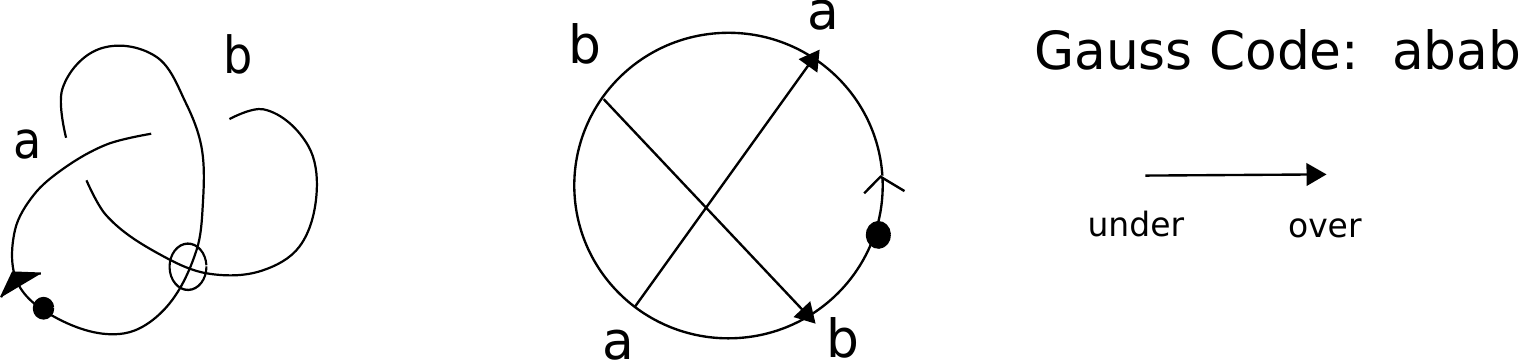}
\end{center}
\caption{The Gauss diagram and Gauss code for the virtualized trefoil}
\end{figure}

Virtual knots and links can be described topologically as embeddings of circles in thickened surfaces (of arbitrary genus) taken up to surface homeomorphisms and 1-handle stabilization.  Virtual crossings are an artifact of the planar representation.  See \cite{AIP} and \cite{IVKT} for more information about this point of view.  \\

Two virtual diagrams are equivalent if and only if there exists a sequence of classical and virtual Reidemeister moves connecting them.  We can describe these moves on planar diagrams or on the Gauss diagrams.  The Reideimester moves on planar diagrams are shown in Figure 3.  Virtual crossings do not appear on Gauss diagrams, so that in this view virtual knots are equivalence classes of Gauss diagrams up to Reidemeister moves.\\

\begin{figure}[h]
\begin{center}
\includegraphics[scale=0.8]{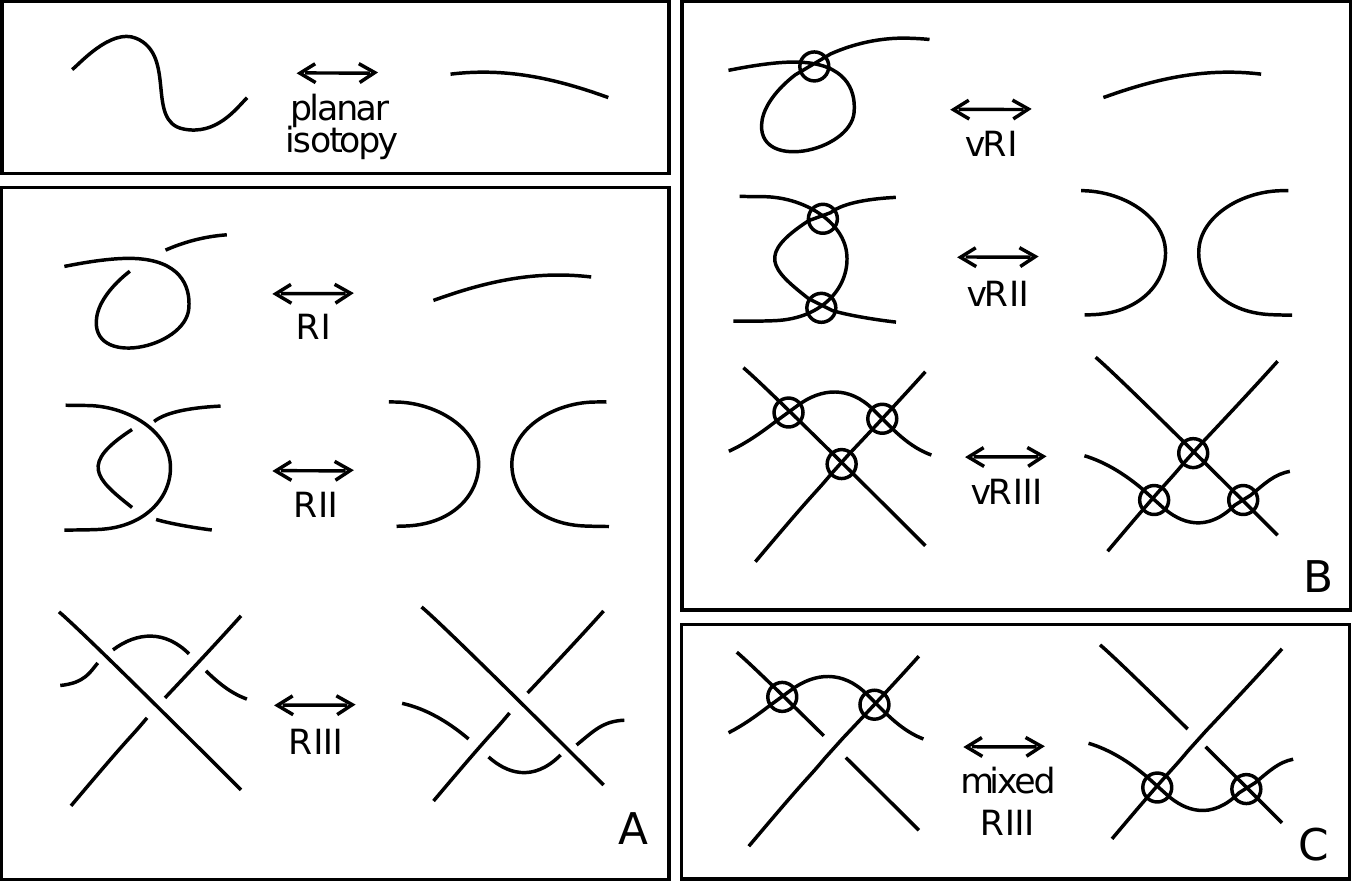}
\end{center}
\caption{Classical and Virtual Reidemeister Moves on Planar Diagrams}
\end{figure}

For an oriented knot diagram, each crossing can be smoothed in two possible ways - an oriented smoothing and an unoriented smoothing, as shown in Figure 4.\\
\begin{figure}[h]
\begin{center}
\includegraphics[scale = 0.7]{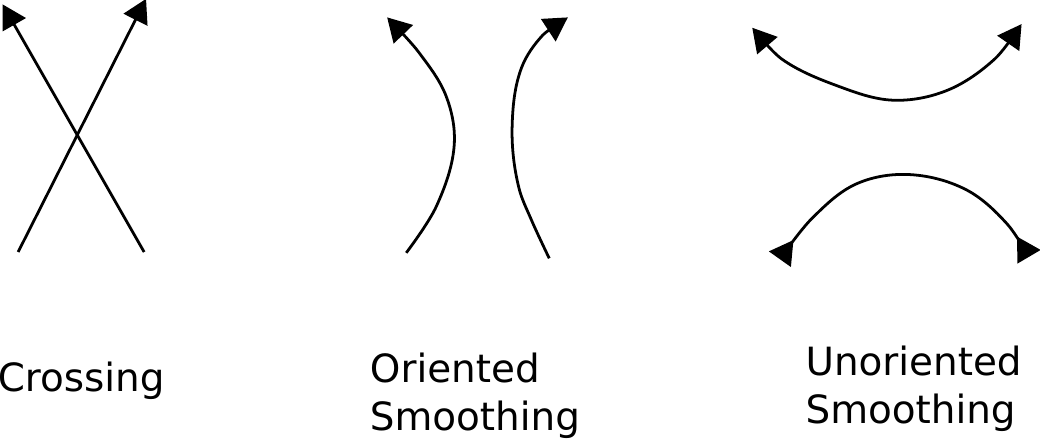}
\end{center}
\caption{An oriented and unoriented crossing smoothing}
\end{figure}

\begin{theorem}
An oriented smoothing of one crossing in a knot (virtual or classical) always gives an oriented link diagram with two components.
\end{theorem}

\begin{proof} This is an easy excercise. Consider what happens to the Gauss diagram of a knot after smoothing one crossing in an oriented manner.
\end{proof}

\subsection{Crossing Parity}
Crossings in a knot diagram can be labeled as even or odd by the definitions given below.
\begin{definition}
 A crossing is an \underline{\textit{even crossing}} if there are an even number of terms in between the two appearances of the crossing in the Gauss Code.  If there are an odd number of terms, the crossing is an \underline{\textit{odd crossing}}.  
\end{definition}

\begin{figure}[h]
\begin{center}
\includegraphics[scale = 0.5]{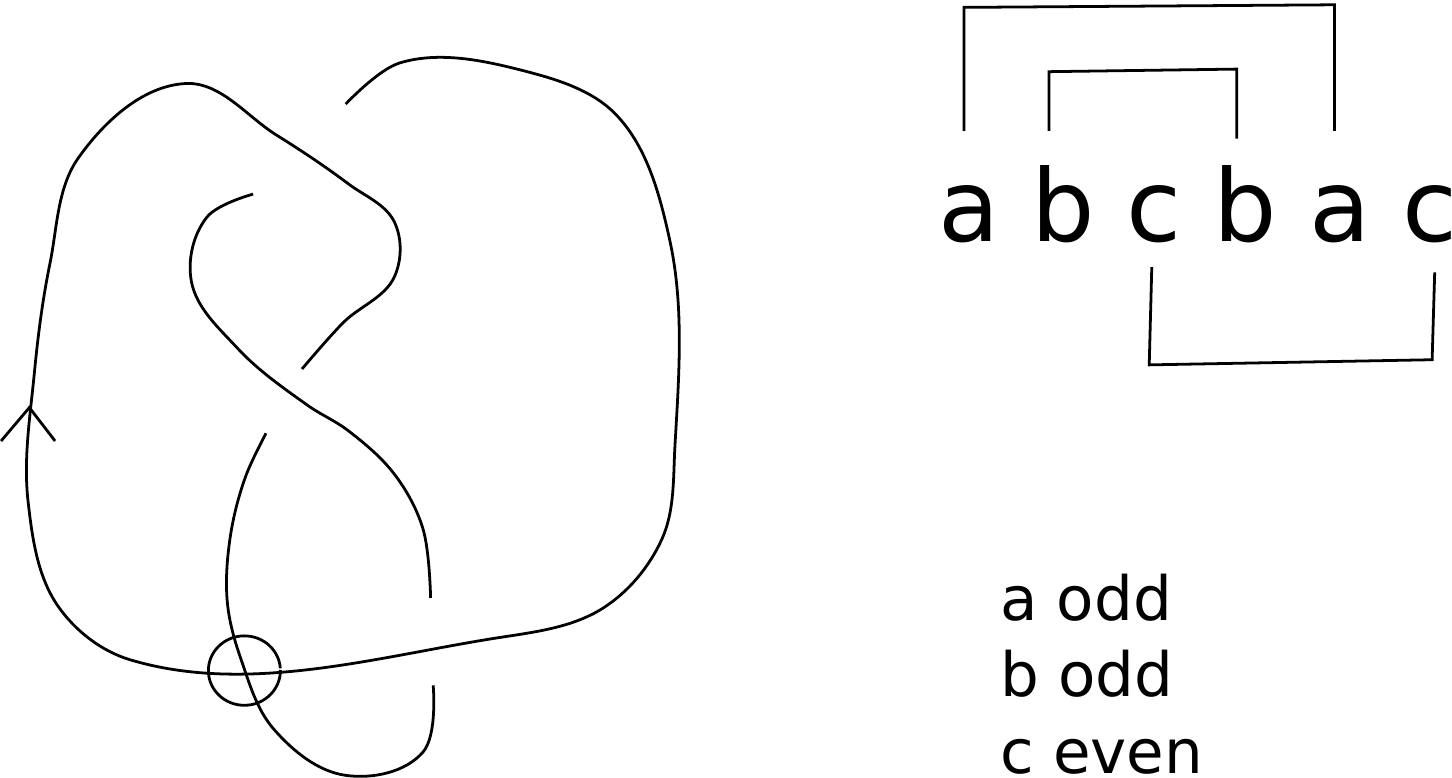}
\end{center}
\caption{The virtual figure eight has even and odd crossings}
\end{figure}

In Figure 5, crossings a and b are odd, and crossing c is even. In classical knots, all crossings are even, but in virtual knots crossings can be even or odd.  One interesting feature about odd crossings is that the \textit{odd writhe}, J(K), \cite{SL} is an invariant of virtual knots, while the writhe is not an invariant of classical knots (it is changed by a Reideimeister 1 move in the diagram). 

\[ 
J(K) = \displaystyle\sum_{c_i \in Odd(C)}w(c_i)
\]

The formula for the odd writhe sums over all crossings in Odd(C), the set of all odd crossings in a knot.  J(K) is the simplest example of an invariant for virtual knots that is obtained by only considering odd crossings in the calculation.  The odd writhe of the knot K in Figure 5 is +2, and its writhe is +1.  The fact that J(K) = +2 proves that K is inequivalent to its mirror image and that K is not classical.\\

\begin{definition}
We call a virtual knot an \underline{\textit{odd virtual}} if all of its crossings are odd.
\end{definition}

\begin{definition}
Let L be an ordered, classical, 2-component link.  Let C be the set of crossings between the 2 linked components (no self-crossings).  When traveling along the first component, \textit{Over(C)} (resp. \textit{Under(C)}) is the set of crossings from C we encounter as overcrossings (resp. undercrossings).  The following are 3 equivalent definitions of \underline{\textit{linking number}} in classical knot theory.\\
1) $lk(L) = \displaystyle\frac{1}{2} \big[$ (\# of positive crossings) - (\# of negative crossings) $\big]$\\
2) $lk(L) = \displaystyle\frac{1}{2} \sum_{c \in C}sign(c)$\\
3) $lk(L) = \displaystyle\sum_{c \in Over(C)}sign(c) = \sum_{c \in Under(C)}sign(c)$\\
\end{definition}

 In Definition 3, the equal sums in part 3 show that ordering the links makes no difference in the classical case when calculating linking number.  For virtual links this is not always true; the two sums in part 3 may be different.  The simplest example of this is the Virtual Hopf Link in Figure 6.  The ordering of the components makes a difference when calculating the sums in part 3 of Definition 3.  Thus in virtual knots, we have two possible linking numbers for an ordered 2-component link.\\

\begin{figure}[h]
\begin{center}
\includegraphics[scale=0.5]{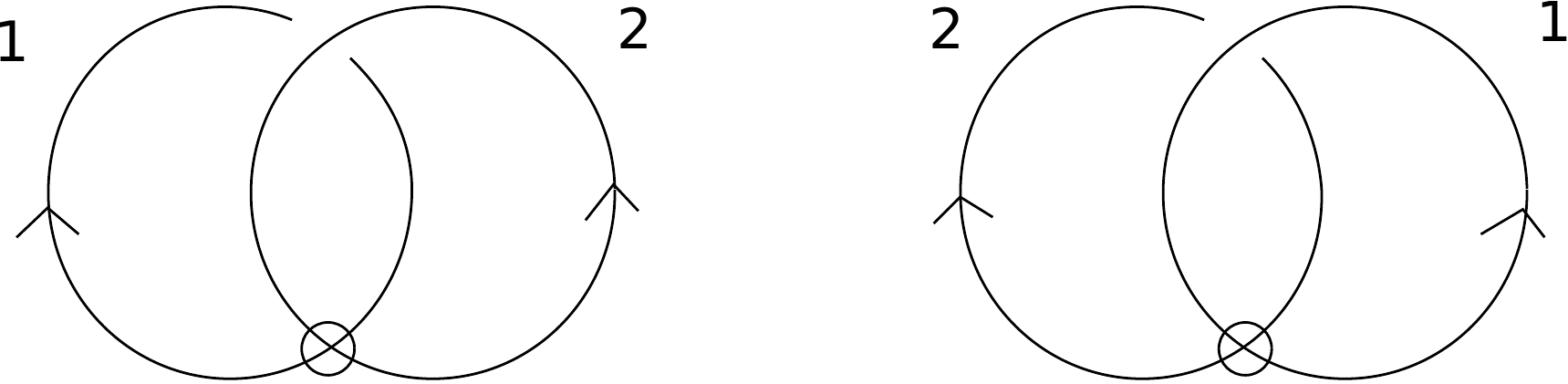}
\end{center}
\caption{Two different orderings of link components in the Virtual Hopf Link}
\end{figure}

\begin{definition}
For an ordered virtual 2-component link we define\\
\begin{center}
\textit{Over linking number} = $lk_O(L) = \displaystyle\sum_{c \in Over(C)}sign(c)$\\
\textit{Under linking number} = $lk_U(L) = \displaystyle\sum_{c \in Under(C)}sign(c)$\\
\end{center}
\end{definition}

Note that Over(C) and Under(C) are defined with respect to the ordering of the link (i.e. from the persepctive of traveling along the first component).  For example, in Figure 6 the Virtual Hopf Link on the left, L, has $lk_O(L) = 0$ and $lk_U(L) = +1$.

\section{The Wriggle Polynomial}

\begin{definition} The \underline{\textit{wriggle number}} for an ordered 2-component oriented link is the difference between the 2 virtual linking numbers.\\
\begin{center}
$W(L) =\displaystyle\sum_{c \in Over(C)} sign(c) - \displaystyle\sum_{c \in Under(C)}sign(c) = lk_O(L) - lk_U(L)$
\end{center}
\end{definition}

The set Over (resp. Under) is the set of crossings between the 2 linked components that we go over (resp. under) while we travel along the first component of the link diagram.  We sum the signs of the crossings in these sets.\\

\begin{theorem}
Wriggle number is a non-trivial invariant of virtual isotopy for ordered links (isotopy where we label the components as first and second) and is trivially 0 for all classical links.\\
\end{theorem}

\begin{proof} This is clear from the definition of wriggle number.
\end{proof}

\begin{figure}[h]
\begin{center}
\includegraphics[scale = 0.6]{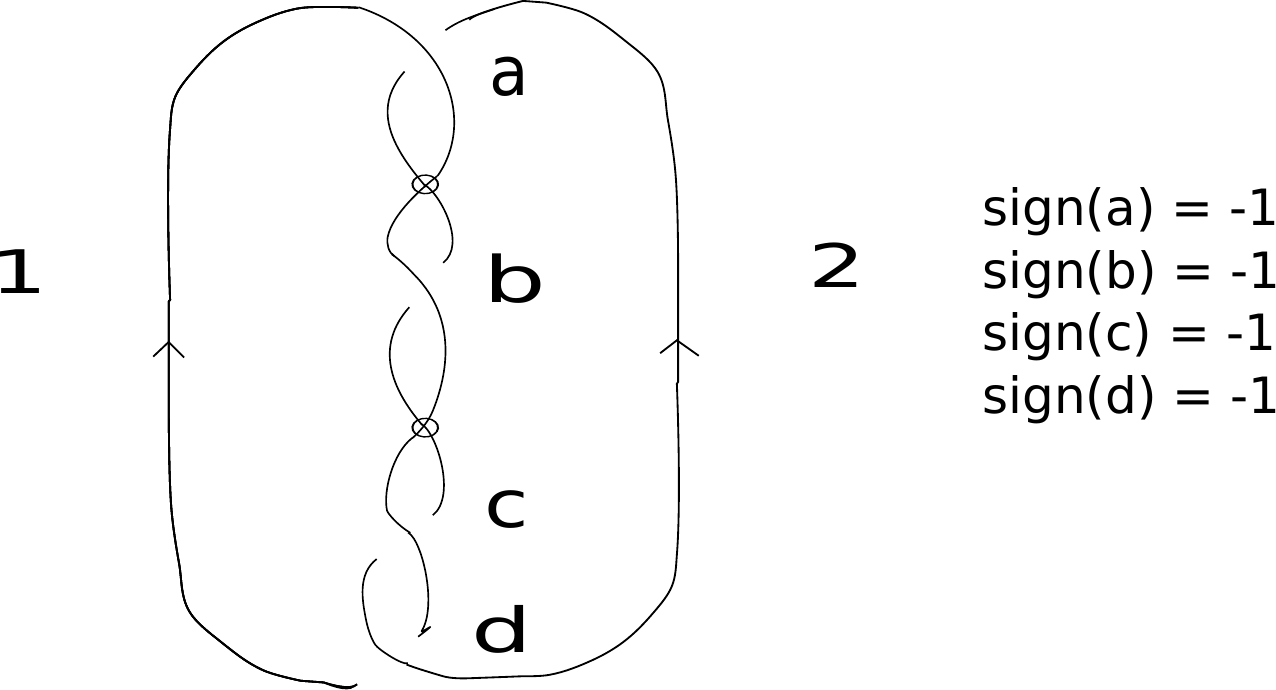}
\end{center}
\caption{Calculating the wriggle number of a virtual link}
\end{figure}

To calculate the wriggle number of the ordered link L in Figure 7, we begin by traveling along the component labeled 1.  As we travel along component 1 crossings a,b,c $\in$ Over(L) and d $\in$ Under(L).  Therefore:\\
\centerline{$W(L) = \big[ (-1) + (-1) + (-1) \big] - \big[ (-1) \big]$ = -2}\\
Notice that if we switch the order of the two components in the link, the sign of each crossing remains the same, but the decision of whether a crossing is ``over" or ``under" is reversed - going over a crossing as you travel along one component results in going under that crossing as you travel along the other component.  The wriggle number of the link in Figure 7 is +2 if we switch the order of the components.\\

\begin{lemma}
Let L be an ordered 2-component oriented link.  Switch the order of the components and call the resulting link $\hat{L}$. Then W(L) = - W($\hat{L}$)\\
\end{lemma}

\begin{proof}  Let C be the set of crossings in L and $\hat{L}$ which are ``linking crossings" - crossings between the 2 components as opposed to within the same component. This set is the same for L and $\hat{L}$ and these are the only crossings which contribute to calculating W(L) and W($\hat{L}$).  Consider what happens at such a crossing when we switch the order of the components of the link, as illustrated in Figure 8.\\

\begin{figure}[h]
\begin{center}
\includegraphics[scale = 0.5]{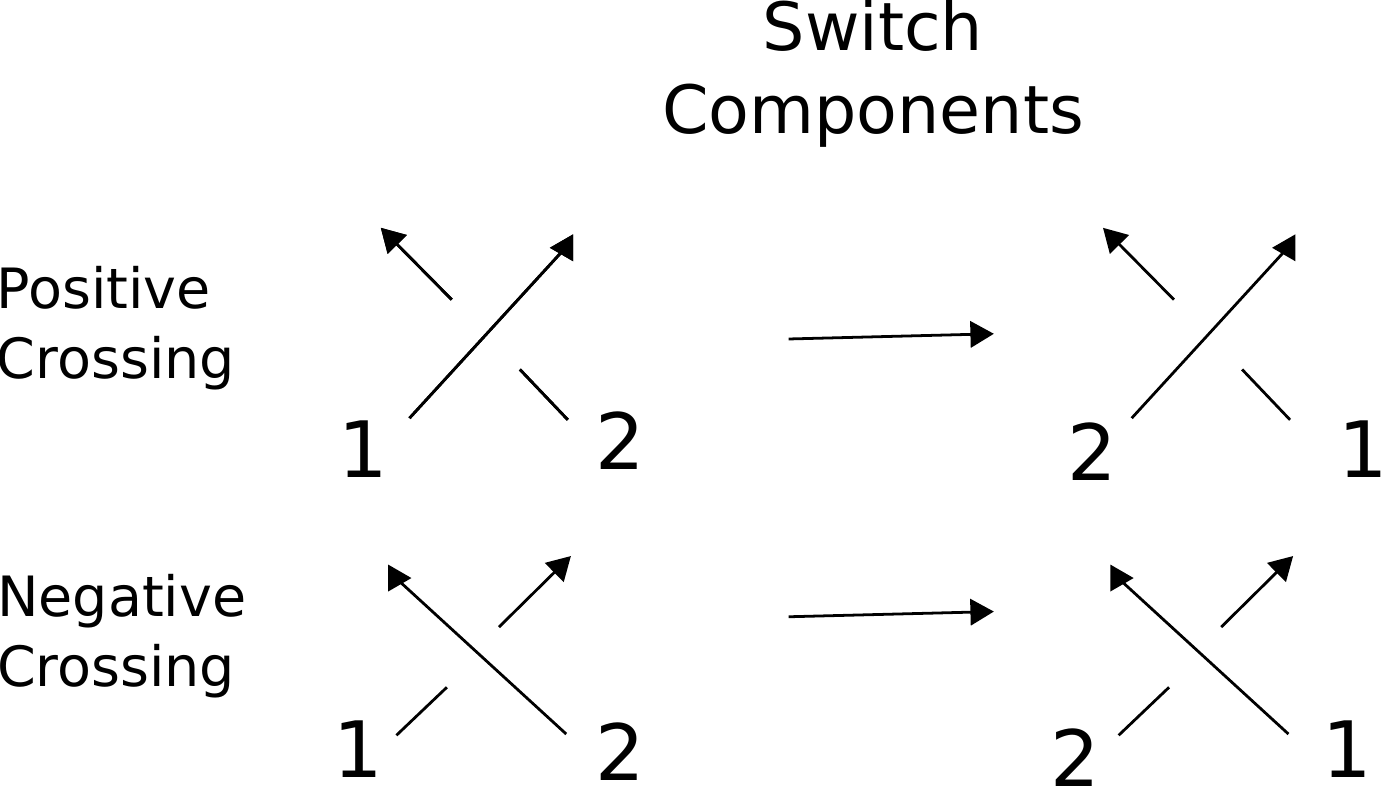}
\end{center}
\caption{Switching the order of the link components}
\end{figure}

The sign of each crossing remains the same.  After switching the order of the components, the difference in the calculation of W(L) and W($\hat{L}$) is that one travels along the other strand for the calculation.  As a result, traveling through an overcrossing while calculating W(L) results in traveling through an undercrossing when calculating W($\hat{L}$) (and vice versa).  Thus W(L) = -W($\hat{L}$)
\end{proof}

\begin{corollary}
$|W(L)|$ is a virtual link invariant.
\end{corollary}

We can obtain a collection of 2-component links from a knot by smoothing classical crossings in an oriented manner.  In Figure 9 we draw the two links we get by smoothing the two classical crossings of the virtualized trefoil. \\

\begin{figure}[h]
\begin{center}
\includegraphics[scale = 0.6]{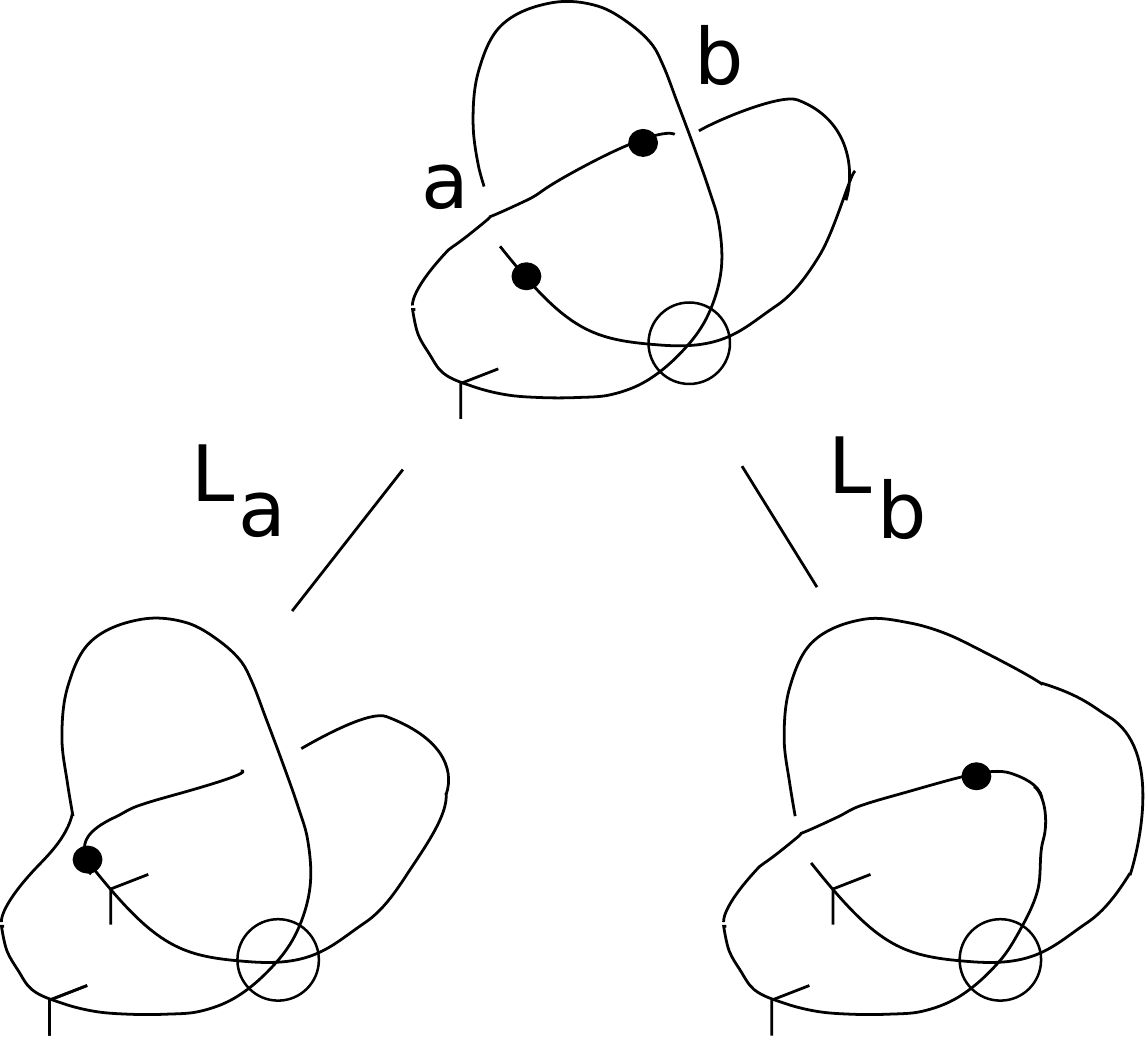}
\end{center}
\caption{Sub-linking diagrams of the virtual trefoil}
\end{figure}

Notice that in Figure 9 we place a dot by each crossing on the incoming under-strand.  Placing this dot in the diagram gives us a technique to consistently order the 2 link components of $L_c$.  After smoothing at that crossing, this dot will appear on one of the link components and we take the component with the dot to be the first component.  At each crossing c we can associate a weight, the wriggle number of the link $L_c$, and create a polynomial as described below. This polynomial, which very naturally can be called the Wriggle Polynomial, will later be shown to equal the Affine Index Polynomial.  To avoid confusion we will usually refer to it as the Affine Index Polynomial, but sometimes we will call it the Wriggle Polynomial.  We shall call it the Wriggle Polynomial in this paper until we prove it is equal to the Affine Index Polynomial.\\

\begin{definition}
The \underline{\textit{Wriggle Polynomial}} of a knot K is\\
\begin{center}
$W_K(t) = \displaystyle\sum_{c \in C}sign(c)t^{W(L_c)} - writhe (K)$
\end{center}
\end{definition}

The summation occurs over all crossings, $c \in C$ of the knot diagram. $L_c$ is the resulting ordered 2-component link that occurs after an oriented smoothing at crossing c.  The order of the linking components is defined by the dot convention explained above.  $W(L_c)$ is then the wriggle number of this sublink. As an example we calculate the Wriggle Polynomial for the virtualized trefoil in Figure 9.\\

\begin{center}
$W_K(t) = sign(a)t^{W(L_a)} + sign(b)t^{W(L_b)} - writhe(K)$\\
= $(+1)t^{-1} + (+1)t^{+1} - 2$\\
=$t^{-1} + t -2$\\
\end{center}

\begin{theorem}
The Wriggle Polynomial, $W_K(t)$, is a virtual knot invariant that is trivially zero on classical knots.
\end{theorem}

\begin{proof}  From the definition it is clear that the Wriggle Polynomial is invariant under all virtual Reidemeister moves.  For classical Reidemeister moves it is enough to check how the polynomial behaves under the 6 oriented moves in Figure 10.  Note that in Figure 10 we give the RIII moves in Gauss diagram form.\\

\begin{figure}[h]
\begin{center}
\includegraphics[scale = 0.8]{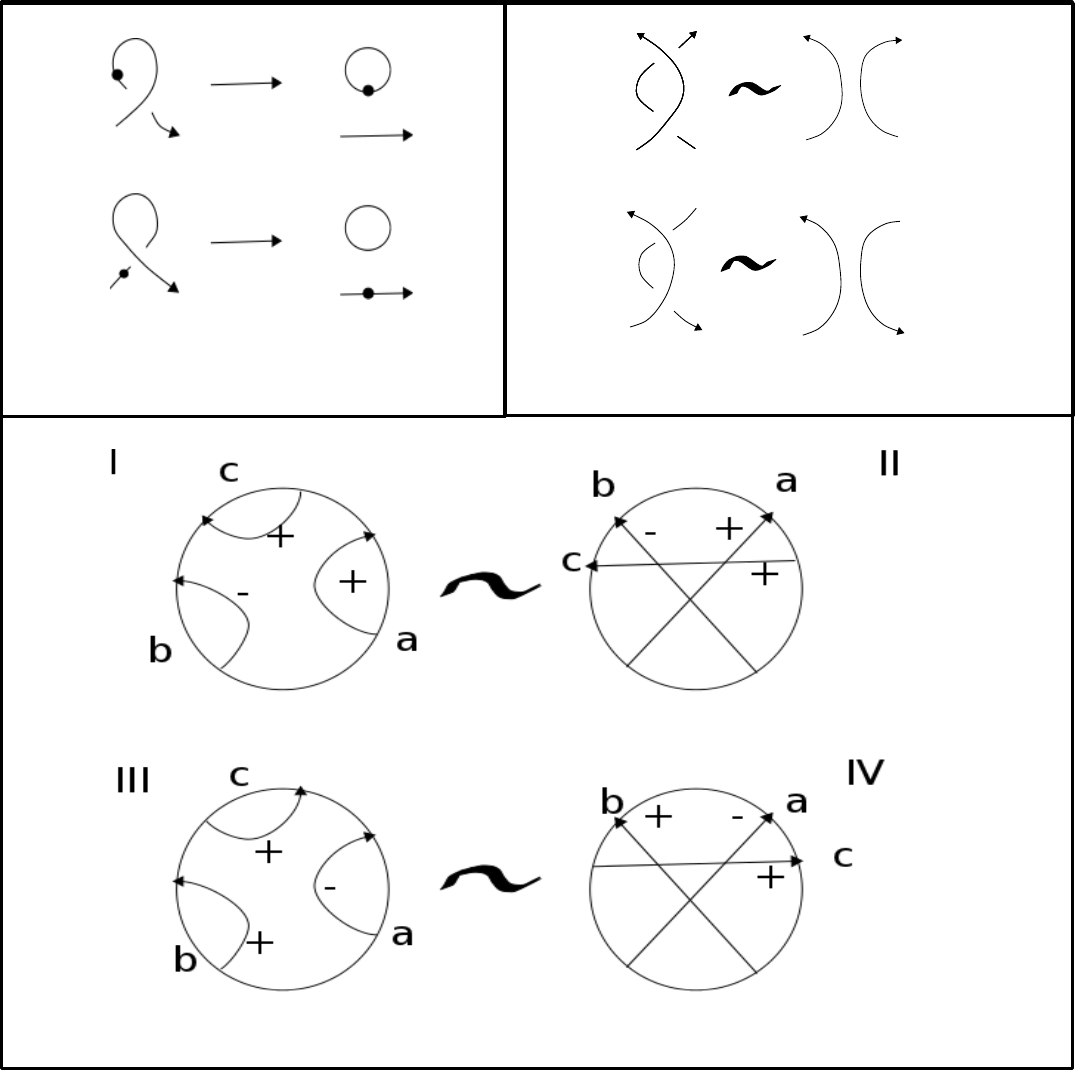}
\end{center}
\caption{A generating set of 6 oriented Reidemeister Moves}
\end{figure}

We already know that writhe(K) is invariant under RII and RIII moves, so let's first check that \\
\centerline{$\displaystyle\sum_{c \in C}sign(c)t^{W(L_c)}$}\\
is invariant under RII and RIII moves.  Recall that is was already shown that the \textit{wriggle number}, W(L), is a knot invariant.  Thus if an RII/RIII move occurs between the 2 link components after smoothing crossing c, $W(L_c)$ remains unchanged (and so does our summation for calculating the polynomial).  All that needs to be checked is the contributions to the polynomial when we smooth crossings at potential RII/RII moves.\\

{\bf\underline{RII Moves - Case I}}
Refer now to Figure 11. The sublink resulting from smoothing crossing a, $L_a$ is on the left, and the sublink resulting from smoothing crossing b, $L_b$, is on the right. Crossing b does not contribute to $W(L_a)$ because it is a self-intersection, and similarly crossing a does not contribute to $W(L_b)$. It is clear that $L_a$ and $L_b$ are the same link with the same order on their components, and so $W(L_a)$ = $W(L_b)$.  When calculating $W_K(t)$, sublinks $L_a$ and $L_b$ have a net contribution of zero in the summation.\\
\centerline{$(+1)t^{W(L_a)} + (-1)t^{W(L_b)} = 0$}\\

\begin{figure}[h]
\begin{center}
\includegraphics[scale = 0.4]{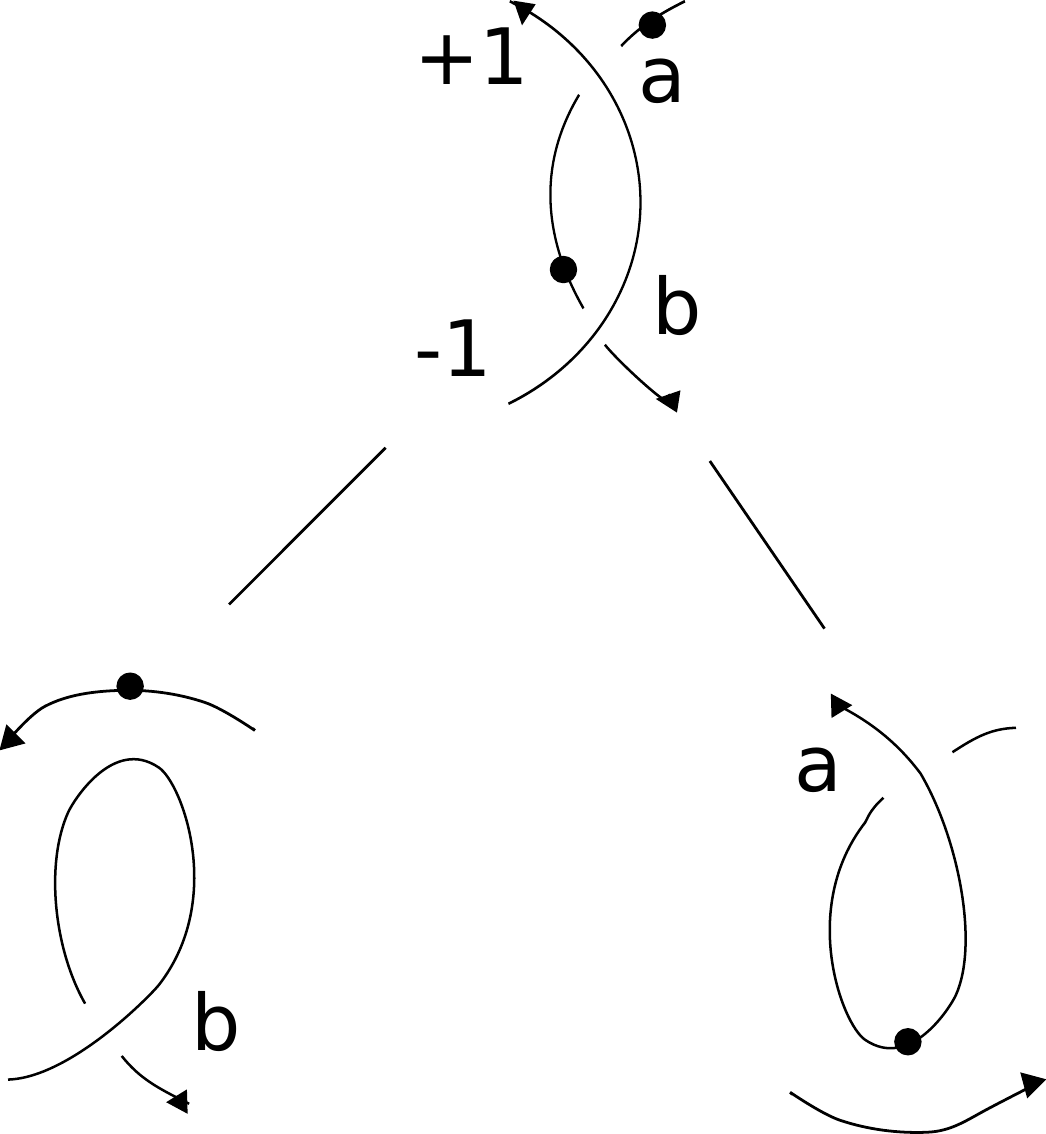}
\end{center}
\caption{Smoothing a crossing in an RII move - Case 1}
\end{figure}

{\bf\underline{RII Moves - Case II}}
Refer now to Figure 12.  $L_a$ is on the left and $L_b$ is on the right. Either both crossings b and a contribute to $W(L_a)$ and $W(L_b)$, resp., or neither does (i.e they are both crossings between the 2 link components, or both self-crossings within 1 link component).  If both are self-crossings, there is no net change to $W_K(t)$.  Consider what happens when they are both crossings between 2 link components.  In calculating $W(L_a)$ we travel under crossing b, which is positive.  In calculating $W(L_b)$ we travel over crossing a, which is negative. In $W(L_a)$ we subtract +1 and in $W(L_b)$ we add -1.  Since the rest of the crossings in $L_a$ and $L_b$ are the same, $W(L_a)$ = $W(L_b)$. The contribution to the summation when calculating $W_K(t)$ is:\\
\centerline{$  (-1)t^{W(L_a)} + (+1)t^{W(L_b)} = 0$}\\

\begin{figure}[h]
\begin{center}
\includegraphics[scale = 0.4]{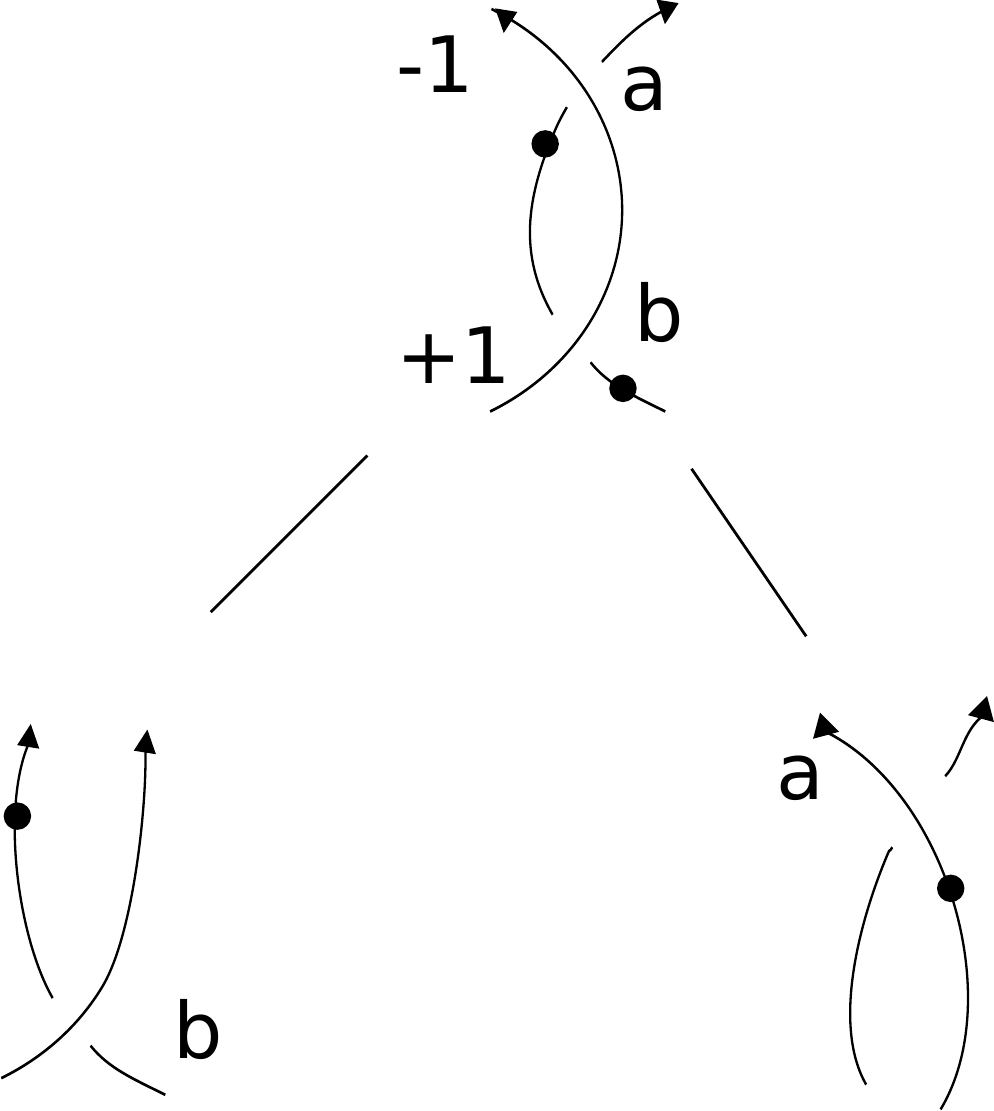}
\end{center}
\caption{Smoothing a crossing in an RII move - Case 2}
\end{figure}

{\bf\underline{RIII Moves}}
\begin{figure}[h]
\begin{center}
\includegraphics[scale = 0.4]{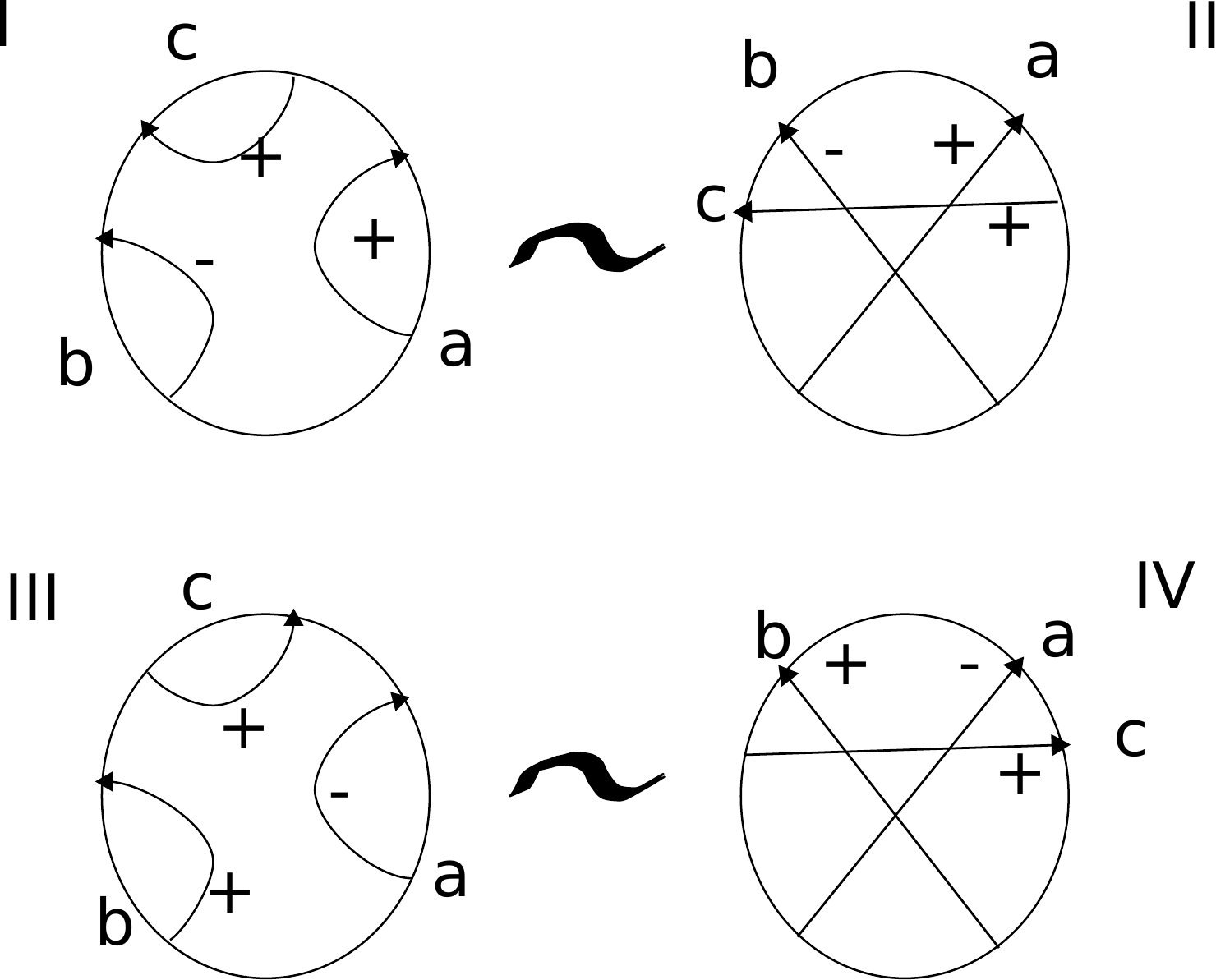}
\end{center}
\caption{Oriented RIII Moves represented on Gauss Diagrams}
\end{figure}

The table below keeps track of the combinatorics of contributions to the calculation of $W_K(t)$ around a Reideimeister III move.  Note that in Figure 10 and Figure 13 we give the RIII move in Gauss Diagram form.  There is no net change to our summation when calculating the polynomial.\\

\begin{center}
\begin{tabular} {|c|c|c|c|c|}
\hline
 & I & II & III & IV \\\hline
$L_a$ & $+t^{W(L_a)}$ & $+t^{W(L_a) + 1 - 1}$  &  $-t^{W(L_a)}$ & $-t^{W(L_a) - 1 + 1}$ \\\hline
$L_b$ & $-t^{W(L_b)}$ & $-t^{W(L_b) + 1 - 1}$  &  $+t^{W(L_b)}$ & $+t^{W(L_a) + 1 - 1}$ \\\hline
$L_c$ & $+t^{W(L_c)}$ & $+t^{W(L_c) - 1 + 1}$  &  $+t^{W(L_c)}$ & $+t^{W(L_c) + 1 - 1}$ \\\hline
\end{tabular}
\end{center}

{\bf\underline{RI Moves}}

\begin{figure}[h]
\begin{center}
\includegraphics[scale = 0.5]{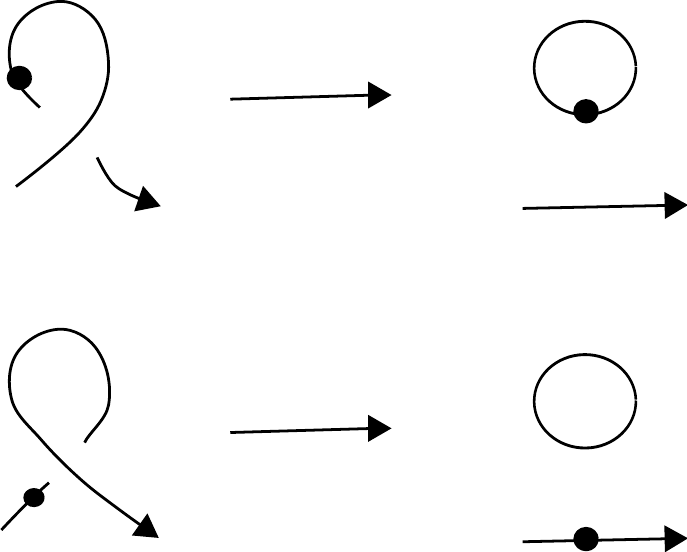}
\end{center}
\caption{Smoothing a crossing in an RI move}
\end{figure}

Refer now to Figure 14.  Smoothing at an RI crossing, c, gives an unlink with $W(L_c)$ = 0, so the sign of the smoothed crossing is contributed to the summation. \\
\centerline{$ sign(c) t^{W(L_c)} = sign(c) t^0 = sign(c)$}

This makes the summation we were investigating\\

\centerline{$\displaystyle\sum_{c \in C}sign(c)t^{W(L_c)}$}

 not invariant under RI.  Now we can normalize the polynomial by subtracting off the writhe(K).  Then the Wriggle Polynomial
\begin{center}
$W_K(t) = \displaystyle\sum_{c \in C}sign(c)t^{W(L_c)} - writhe (K)$
\end{center}
is invariant under classical RI, RII, and RIII moves.

Since the polynomial can be calculated from the Gauss diagram, we should note here that the calculation is not affected by any virtual moves, making the polynomial a virtual knot invariant.  The poynomial is trivially zero on classical knots as a consequence of the wriggle number being constantly zero for all classical links.
\end{proof}

\begin{theorem}
Consider a knot $K$, its inverse, $K^-$, its mirror image $K^*$, and the connected sum $K$\#$K'$.\\
1) $W_K(t) = W_{K^-}(t^{-1})$\\
2) $W_K(t) = -W_{K^*}(t)$\\
3) $W_{K \# K'}(t) = W_K(t) + W_{K'}(t)$\\
\end{theorem}

Notice that if we restrict the Wriggle Polynomial so that it only sums over the odd crossings in a diagram, the writhe(K) no longer has to be subtracted to normalize the polynomial (since RI crossings are even).\\

\begin{definition}
The \underline{\textit{Odd Wriggle Polynomial}}, $\hat{W}_K(t)$, is defined as\\
\begin{center}
$\hat{W}_K(t) = \displaystyle\sum_{c \in Odd(C)}sign(c)t^{W(L_c)}$
\end{center}
where the summation occurs over all odd crossings of the diagram.
\end{definition} 

The Odd Wriggle Polynomial can be shown to be equivalent to the Odd Writhe Polynomial defined by Cheng \cite{OW}.  Equality occurs after multiplying the Odd Writhe Polynomial by $t^{-1}$.

\begin{theorem}
The Odd Wriggle Polynomial is an invariant of virtual knots.
\end{theorem}

\begin{proof} It is easy to check that an RIII move does not affect whether a crossing is even or odd, and thus will not affect the calculation of the polynomial.  In an RII move, the 2 crossings involved are either both even or both odd, and so the polynomial remains unchanged.
\end{proof}

\section{Cosmetic Crossing Change Conjecture}
\begin{definition}
A \underline{\textit{nugatory}} crossing in a knot (coined by P.G. Tait) is a crossing that can be untwisted from the diagram.  An example of a nugatory crossing is one that can be undone with an RI move.  Figure 15 illustrates the most general form for a nugatory crossing.
\end{definition}

\begin{figure}[h]
\begin{center}
\includegraphics[scale = 0.3]{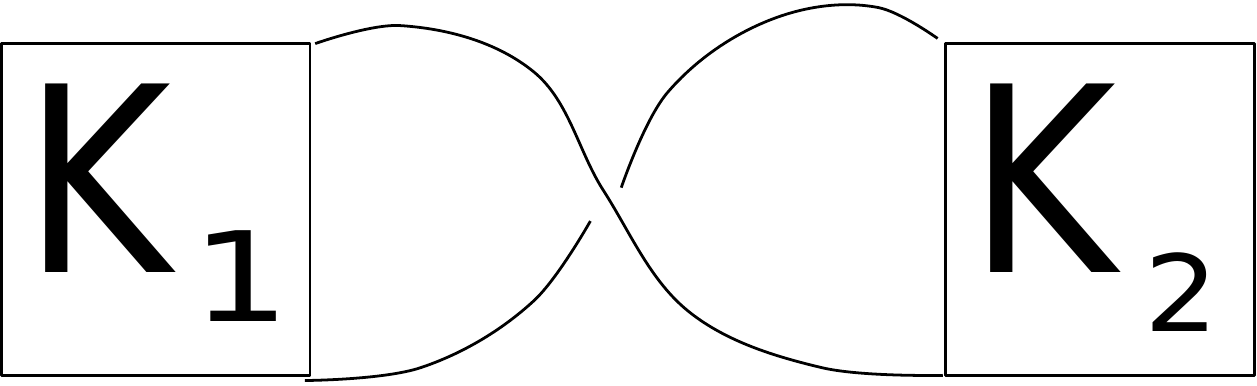}
\end{center}
\caption{A nugatory crossing}
\end{figure}

\begin{definition}
A \underline{\textit{crossing change}} in a knot diagram $D$ switches a positive crossing to a negative crossing (and vice versa).   A crossing change is said to be \underline{\textit{cosmetic}} if the new diagram, $D^{\prime}$, is isotopic (classically or virtually) to $D$.
\end{definition}

It is clear that all crossing changes on nugatory crossings are cosmetic, and these are called trivial crossing changes.\\

\underline{\textbf{Open Question:}}  Do non-trivial cosmetic crossing changes exist? The question was stated as problem 1.58 of Kirby's Problem List, and attributed to Xiao-Song Lin \cite{Kirby}.   There exist examples of $K_1$ and $K_2$ that differ by a single crossing change and are equivalent but their orientation is reversed \cite{Effie}.  For example, $K_1$ = P(3,1,-3) and $K_2$ = P(3, -1, -3).\\

The general feeling is that no non-trivial cosmetic crossing changes exist.  The work so far has shown this for classes of classical knots.  We will prove this conjecture for some classes of virtual knots.\\
\begin{center}
\begin{tabular}{| c | l |}
\hline
The Unknot & Scharlemann-Thompson (CMH, 87) \\\hline
Two Bridge Knots & I.Torisu (TAIA, 97) \\\hline
Fibered Knots & Kalfagianni (Crelle, 11) \\\hline
Genus one knots & Kalfagianni, Balm, Friedl, Powell (2012) \\\hline
\end{tabular}
\end{center}

\begin{theorem}
Odd Virtual Knots do not admit cosmetic crossing changes.
\end{theorem}

\begin{proof}  Recall that an odd virtual knot has only odd crossings.  Let $K_+$ and $K_-$ be two knots that differ from each other by a sign change at some odd crossing.  Without loss of generality, assume $K_+$ is postive at the respective crossing and $K_-$ is negative.  We will show $K_+$ is not isotopic to $K_-$ by considering the Odd Wriggle Polynomial and showing that $\hat{W}_{K_+}(t) - \hat{W}_{K_-}(t) \neq 0$.\\

Symmetrically label the odd crossings $c_1, ... , c_n$ in $K_+$ and $\hat{c}_1, ... , \hat{c}_n$ in $K_-$, with $c_1$ being the crossing that switched from positive to negative.\\
\centerline{$C = \{ c_1, ... , c_n \} \ \ \ \ \ \hat{C} = \{ \hat{c}_1, ... , \hat{c}_n \} $}\\
\centerline{$sign(c_1) = +1 \ \ \ \ \ \ sign(\hat{c}_1) = -1$}\\ 
$\hat{W}_{K_+}(t) - \hat{W}_{K_-}(t) = t^{W(L_{c_1})} + t^{W(L_{\hat{c}_1})} + \displaystyle\sum_{i = 2}^n sign(c_i)t^{W(L_{c_i})} - sign(\hat{c}_i)t^{W(L_{\hat{c}_i})}$\\
Note that  $sign(c_i) = sign(\hat{c}_i)$ for i = 2, ..., n.\\

\underline{\textit{Claim 1}} $W(L_{c_1}) = -W(L_{\hat{c}_1})$\\
\textit{Proof of Claim 1.} Since $K_+$ and $K_-$ differ only by a crossing switch at $c_1$ (resp. $\hat{c}_1$), smoothing the crossing $c_1$ in $K_+$ results in the same link as smoothing the crossing $\hat{c}_1$ in $K_-$.  Recall the we calculate the wriggle number by ordering the link obtained from the smoothed crossing by placing a dot on the understrand before entering the crossing.  After smoothing, the strand with the dot will be the first component.  \\

\begin{figure}[h]
\begin{center}
\includegraphics[scale = 0.5]{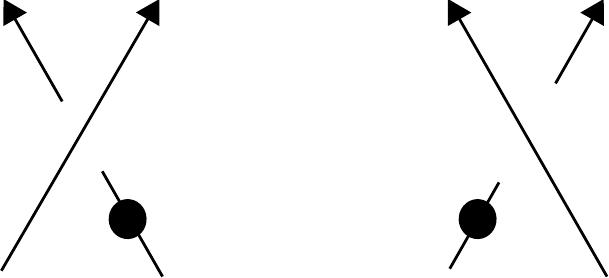}
\end{center}
\caption{A crossing change places the dot on the other component}
\end{figure}

Figure 16 shows switching a crossing from positive to negative and then smoothing results in the dot being placed on the other link component.  So $L_{c_1}$ and $L_{\hat{c}_1}$ are the same link, with the order of the components switched.  By Lemma 1, we know that W($L_{c_1}$) = -W($L_{\hat{c}_1}$).  To simplify notation, from this point on we will let W($L_{c_1}$) = k.\\

\underline{\textit{Claim 2.}} $W(L_{c_i}) = W(L_{\hat{c}_i})$ for i = 2, ... , n \\
\textit{Proof of Claim 2.}  We must consider how the contribution of $c_1$ and $\hat{c}_1$ affect $W(L_{c_i})$ and $W(L_{\hat{c}_i})$, respectively.  Note that $c_1$ contributes to $W(L_{c_i})$ iff $\hat{c}_1$ contributes to $W(L_{\hat{c}_i})$ (since $L_{c_i}$ and $L_{\hat{c}_i}$ are the same link, with opposite signs on crossings $c_1$ and $\hat{c}_1$).  If $c_1$  and $\hat{c}_1$ are self-crossings in links $L_{c_i}$ and $L_{\hat{c}_i}$  we have $W(L_{c_i}) = W(L_{\hat{c}_i})$ and the net contribution to the summand\\

\[\displaystyle\sum_{i=2}^n sign(c_i)[t^{W(L_{c_i})} - t^{W(L_{\hat{c}_i})}]\] is zero.\\  

Without loss of generality, we may assume that the crossings $c_1$ and $\hat{c}_1$ contribute to the summand.

\begin{center}
$\hat{W}_{K_+}(t) - \hat{W}_{K_-}(t) = t^k + t^{-k} - \displaystyle\sum_{i=2}^n sign(c_i)[t^{W(L_{c_i})} - t^{W(L_{\hat{c}_i})}]$
\end{center}

Now $c_1$ contributes to W($L_{c_i}$) a positive sign (added if we pass over $c_1$ and subtracted if we pass under $c_1$).  Notice that $\hat{c}_1$ contributes a negative sign to W($L_{\hat{c}_i}$) (added if we pass over $\hat{c}_1$ and subtracted if we pass under $\hat{c}_1$).  Traveling over $c_1$ means we travel under $\hat{c}_1$.  Therefore in W($L_{c_i}$) the positive sign is added (or subtracted) and in W($L_{\hat{c}_i}$) the negative sign is subtracted (or added), making W($L_{c_i}$) = W($L_{\hat{c}_i}$).\\

So W($L_{c_i}$) = W($L_{\hat{c}_i}$) for i = 2, ... , n.\\

Therefore, letting W($L_{c_1}$) = k we get
\begin{center}
$\hat{W}_{K_+}(t) - \hat{W}_{K_-}(t) = t^k + t^{-k} \neq 0$
\end{center}

And so $K_+$ is not isotopic to $K_-$.
\end{proof}

We now ask what happens if a cosmetic crossing change occurs at an even crossing in a virtual knot.  Calculating the difference of the Wriggle Polynomial for the knots results in:\\
\centerline{$W_{K_+}(t) - W_{K_-}(t) = t^k + t^{-k} - 2$}\\
This polynomial is zero iff k = 0, where k is the wriggle number of the sublink obtained by smoothing the crossing where the sign change occurred.\\

\begin{definition}
A crossing c in a knot is called a \underline{\textit{classical crossing}} if $W(L_c)$ = 0.   If $W(L_c) \neq 0$, c is called a \underline{\textit{non-classical crossing}}.
\end{definition}

Very appropriately, all crossings in classical knots are ``classical crossings".  Crossings in virtual knots can be labeled labeled classical or non-classical.  There exist examples of virtual knots that have all non-classical crossings, all classical crossings, and mixed classical/non-classical crossings.

\begin{definition}
A virtual knot is called a \underline{\textit{pure virtual knot}} if all its crossings are odd and/or non-classical.  For comparison, in classical knots all crossings are even and classical.
\end{definition}

\begin{corollary}
Pure Virtual Knots do not admit cosmetic crossing changes.
\end{corollary}

By the corollaries, only crossings that are even AND classical in a virtual knot may possibly admit cosmetic crossing changes.\\

\section{The Wriggle Polynomial equals the Affine Index Polynomial}
We shall now consider the definition of the Affine Index Polynomial, as given by Kauffman \cite{AIP} and show the two definitions produce equal polynomials, which is not immediately clear.  We will refer to our definition as ``the linking number definition" and the new definition as ``the algebraic definition".  In calculating the Affine Inde Polynomial, we begin by labeling the arcs in the knot diagram as described in Figure 17.\\

\begin{figure}[h]
\begin{center}
\includegraphics[scale=0.5]{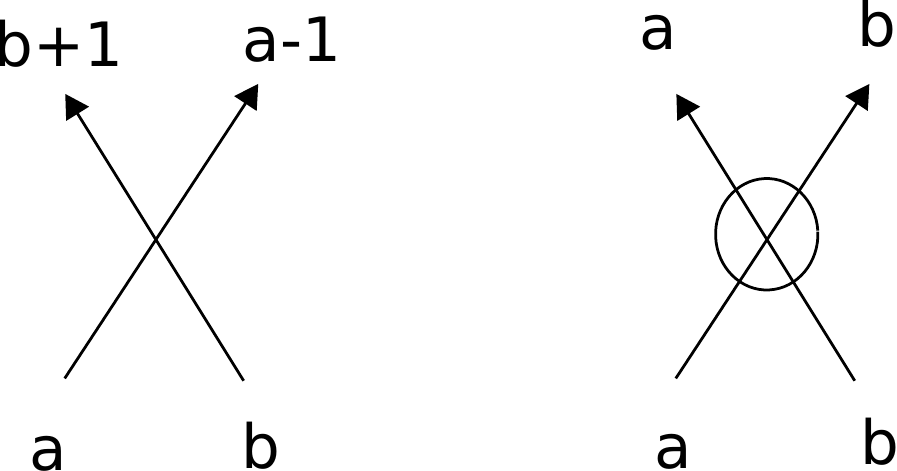}
\end{center}
\caption{Rules for labeling the knot diagram}
\end{figure}

To start the labeling of the knot diagram, choose any arc to label with 0.  To form the polynomial, at each classical crossing obtain a number as follows:\\
\centerline{$ W_+(c) = a - (b+1) \ \ OR \ \ \ W_-(c) = b - (a-1)$}\\
Which calculation is chosen depends on whether the crossing is positive or negative. The algebraic definition of the Affine Index Polynomial is given by this formula.\\
\begin{center}
$W_K (t) = \displaystyle\sum_{c \in C}sign(c)t^{W_{\pm}(c)} - writhe(K)$
\end{center}

\begin{theorem}
The Wriggle Polynomial equals the Affine Index Polynomial.
\end{theorem}

\begin{proof} It is enough to show equality of the exponents, i.e. $W(L_c) = W_{\pm}(c)$.  We do this for the positive crossing by carefully analyzing each definition of weight on a Gauss diagram, and rewriting the definition of both weights, $W(L_c)$ and $W_{\pm}(c)$, as a new definition. Repeating the exact same analysis on the negative crossing completes the proof.  Figures 18 shows how the algebraic labeling of the knot diagram is transfered to a Gauss diagram.\\

\begin{figure}[h]
\begin{center}
\includegraphics[scale = 0.5]{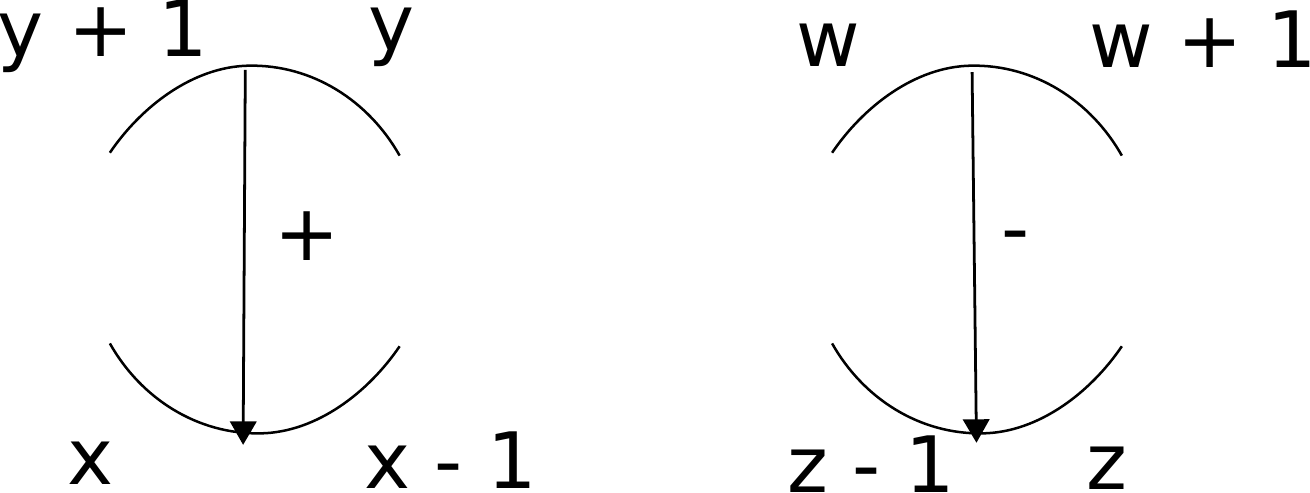}
\end{center}
\caption{Labeling rules for the Gauss Diagram}
\end{figure}


Refering to the Gauss diagram in Figure 18, the weight assigned to each crossing (chord) in the algebraic definition of the Affine Index Polynomial, $W_{\pm}(c)$, is defined as follows.

\[ 
W_{\pm}(c_i) = 
\begin{cases} 
x - (y+1) & \mbox{ if } sign(c_i) = +1\\
z - (w+1) & \mbox{ if } sign(c_i) = -1\\
\end{cases}
 \]

{\bf\underline{Section 1} The Linking Number Definition: Analyzing $W(L_c)$}\\
Recall that in the linking number definition of the Wriggle Polynomial, $W(L_c)$ equals the wriggle number of the sublink we get when smoothing at crossing c. \\

\centerline{$W(L_c) = lk_O(L_c) - lk_U(L_c) = \displaystyle\sum_{c \ \in \ Over}sign(c) - \sum_{c \ \in \ Under}sign(c)$}

\begin{figure}[h]
\begin{center}
\includegraphics[scale = 0.5]{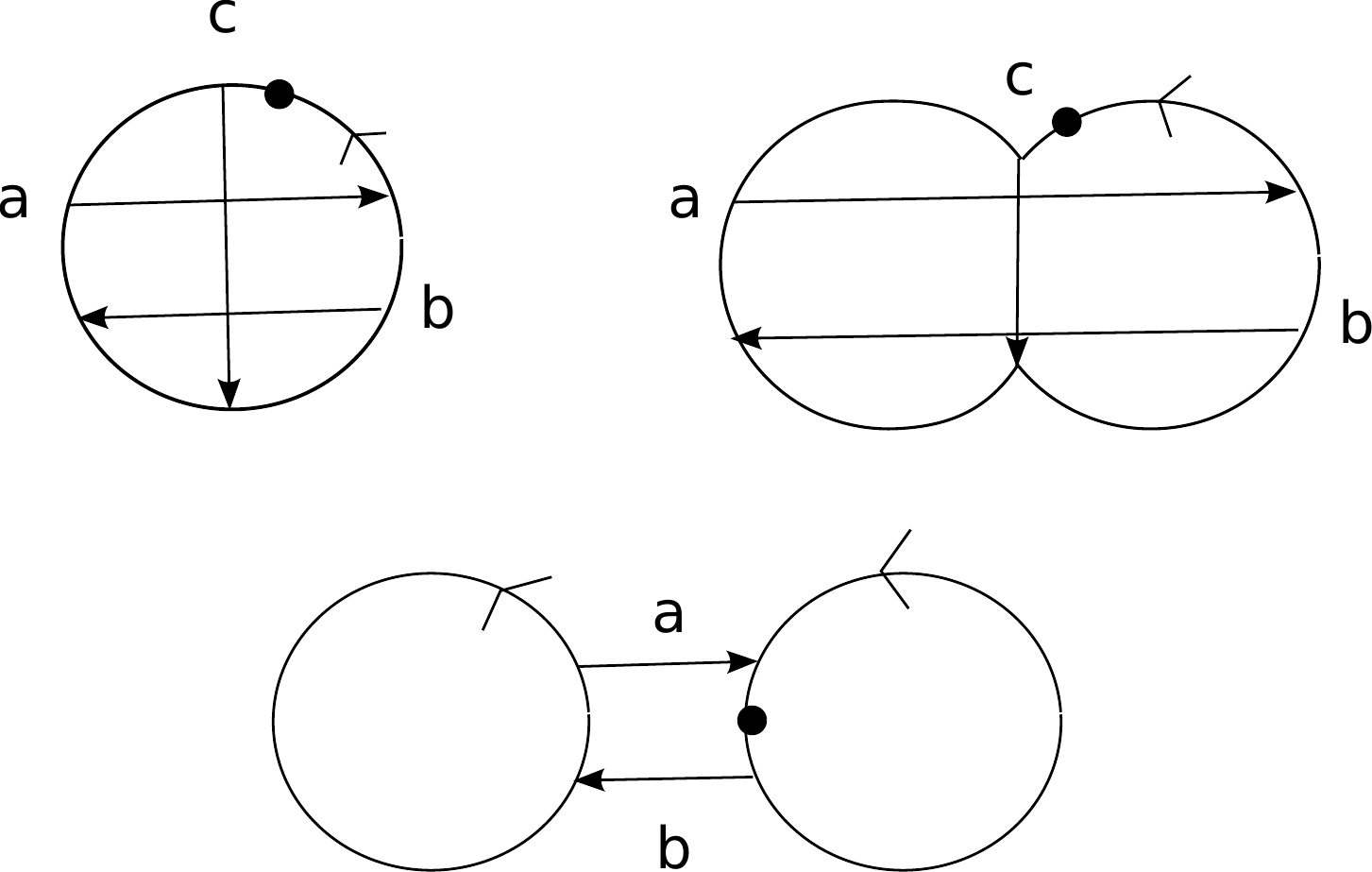}
\end{center}
\caption{Calculating the wriggle number from the Gauss Diagram}
\end{figure}

Consider how the calculation of $W(L_c)$ appears on the Gauss diagram.  Figure 19 expresses what happens on the Gauss diagram when we do an oriented smoothing of crossing c in a knot.  Traveling around the link component that has inherited the dot, we encounter ``linking chords", which are used in the calculation of the wriggle number.  In Figure 19, notice that these ``linking chords" are precisely the chords which crossed chord c before the smoothing occurred.  Thus, in the Gauss diagram, chords crossing chord c are the only chords which contribute to $W(L_c)$.\\

\begin{definition}
For a distinguished chord y in a Gauss Diagram, chords crossing y intersect it \underline{positively} and \underline{negatively} as in Figure 20.
\end{definition}

\begin{figure}[h]
\begin{center}
\includegraphics[scale = 0.4]{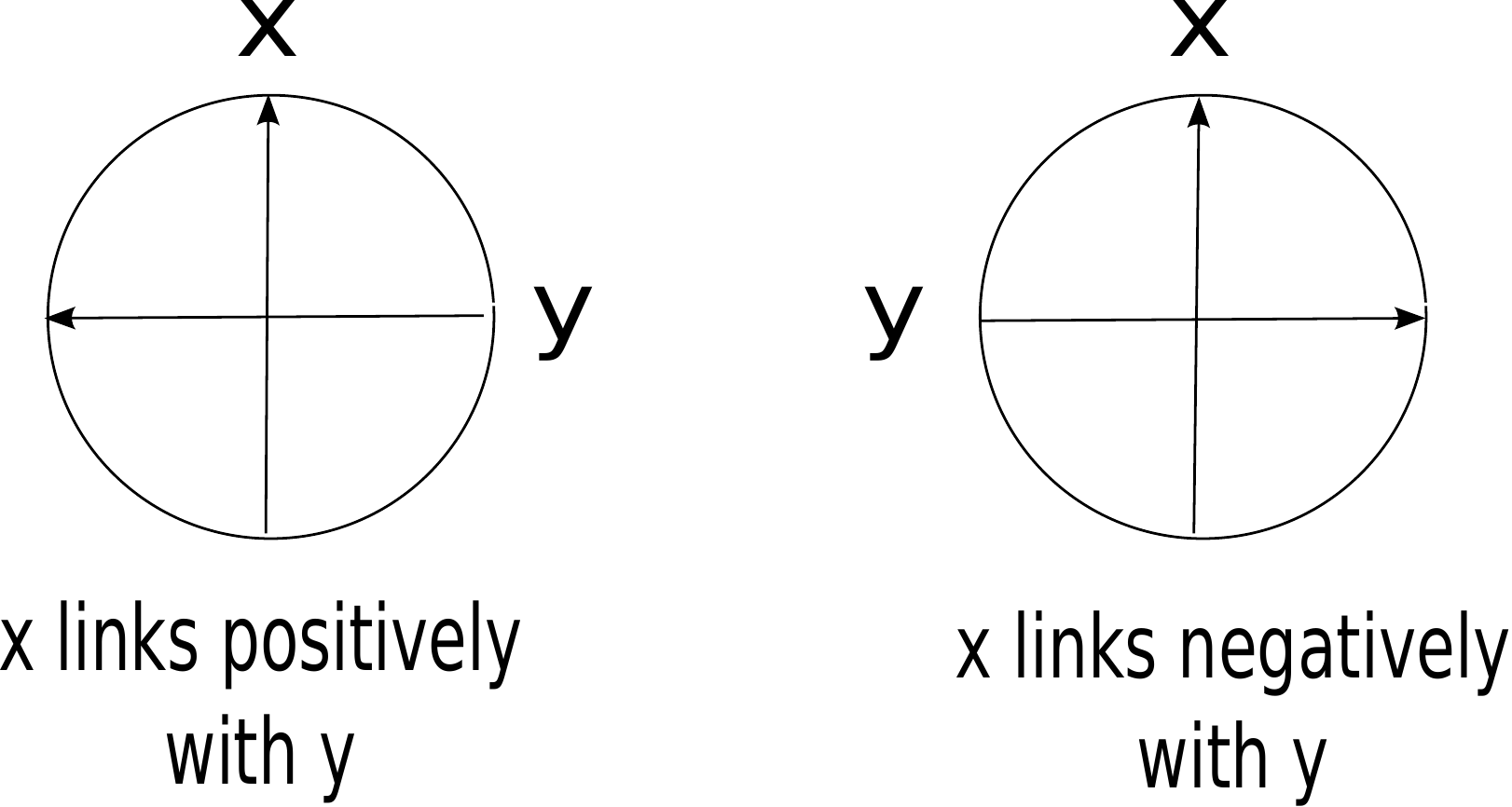}
\end{center}
\caption{Defining +/- intersection between 2 chords in a Gauss Diagram}
\end{figure}

For a distinguished chord c let:\\
\centerline{$P_c$ = \{ the set of chords which c intersects positively \} }
\centerline{$N_c$ = \{ the set of chords which c intersects negatively \} }

The definition of $W(L_c)$ can now be rewritten as:\\
\centerline{$W(L_c) = \displaystyle\sum_{c \ \in \ P_c}sign(c) - \sum_{c \ \in \ N_c}sign(c) $}
Using the absolute value of this definition of $W(L_c)$ to create a polynomial was first defined by A. Heinrich \cite{Hen}.\\

{\bf\underline{Section 2} The Algebraic Definition: Analyzing $W_{\pm}(c)$}\\
We now show that $W_{\pm}(c)$, in the algebraic definition of the Affine Index Polynomial,  can be rewritten the same way.  
\[
W_{\pm}(c) = \displaystyle\sum_{c \ \in \ P_c}sign(c) - \sum_{c \ \in \ N_c}sign(c)
\]
We begin with a lemma from Kauffman \cite{AIP}.\\

\begin{lemma}(Kauffman \cite{AIP})
When calculating the algebraic definition of the Affine Index Polynomial, labeling arcs in the following two ways results in equal values when calculating $W_{\pm}(c)$ for each crossing.\\
1) Arbitrarily choosing an arc to label 0, and then following the labeling schema in Figure 17.\\
2) Labeling an arc in the oriented knot diagram by the sum of the signs of crossings first encountered as ovecrossings after traveling around the the knot (starting from the chosen arc).\\
\end{lemma}

The labeling system defined in part 2 of the lemma was first defined by Z. Cheng \cite{OW} and a different weight was calculated.  As a consequence of this lemma, we may assume that each arc in the Gauss diagram is labeled by the the sum of signs of chords first encountered as over (i.e. the head of the chord) after traveling around the circle in the Gauss diagram.\\

A more detailed description of this alternate labeling method on the Gauss diagram is:\\

1) Pick an arbitrary point on the circle which lies in some arc, $a_i$.\\
2) From that point, travel around the circle counterclockwise.  Sum the signs of the chords where you encounter the over part of a crossing before the under part.  In other words, sum the signs of the chords where, when traveling around the circle, you encounter the head of the chord before you encounter the tail.\\
3) This sum is the integer you assign to the arc $a_i$.\\

For a distinguished chord c (c positive), place a dot on the incoming under-strand and look at the local labeling of c on the Gauss diagram.  Recall that for a positive crossing, $W_{\pm}(c) = x - (y+1)$\\

\begin{figure}[h]
\begin{center}
\includegraphics[scale = 0.4]{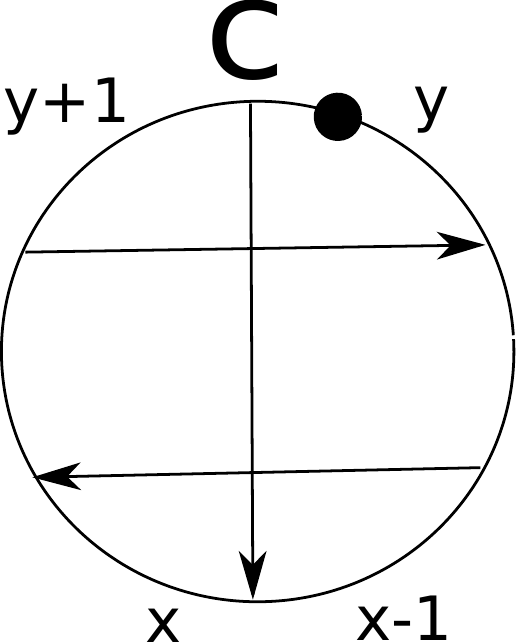}
\end{center}
\end{figure}

The number x comes from starting at that arc in the Gauss diagram, and traveling along the circle counterclockwise, and summing the signs of the crossing we encountered as overcrossings.  Similarly, the number y comes from starting at that arc in the Gauss diagram, traveling along the circle counterclockwise, and summing the signs of the overcrossings we encounter.  Figure 21 shows what types of chords we might have to consider when calculating x and y.\\

\begin{figure}[h]
\begin{center}
\includegraphics[scale = 0.6]{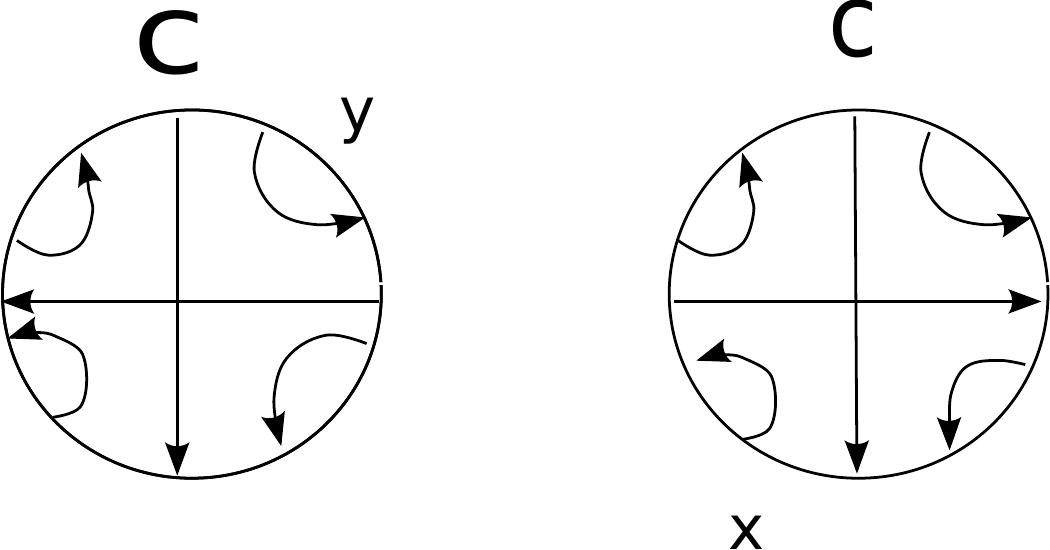}
\end{center}
\caption{Types of chords to consider when calculating x and y}
\end{figure}

It is clear that the set of chords we consider in each diagram that do not cross c are the same.  When taking the difference x - y we end up taking the difference of the the sums of signs of the the of the linking chords drawn in Figure 21.\\

Therefore we get the equation:\\
\centerline{$W_{\pm}(c) = \displaystyle\sum_{c \ \in \ P_c}sign(c) - \sum_{c \ \in \ N_c}sign(c)$}
Now we have obtained $W(L_c) = W_{\pm}(c)$.
\end{proof}

\begin{definition}
Call a labeling of the arcs around a crossing a \underline{classical labeling} if the arcs are labeled as follows
\end{definition}
\begin{figure}[h]
\begin{center}
\includegraphics[scale=0.5]{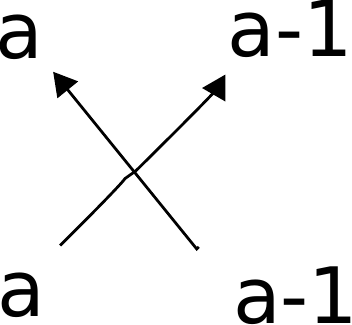}
\end{center}
\end{figure}

\begin{lemma}
A crossing c admits a classical labeling iff $W(L_c)$ = 0.
\end{lemma}

\begin{corollary}
A crossing in a virtual knot cannot admit a cosmetic crossing change if it is odd or if it has a non-classical labeling.  Crossings that are even and have a classical labeling might admit cosmetic crossing changes.
\end{corollary}

\begin{proof} It is clear from the definitions that when calculating the wriggle polynomial, a crossing has wriggle number zero iff the labeling on the crossing is classical.
\end{proof}

\section{Results on Mutation}
The Affine Index Polynomial (aka the Wriggle Polynomial) has some success in detecting mutant virtual knots.  Recall that a Conway mutation of a knot cuts out a tangle, L, from the knot diagram and reglues that tangle after a horizantal flip, a vertical flip, or rotation by $180^{\circ}$.  These 3 types of mutation are shown in Figure 22.  

\begin{figure}[h]
\begin{center}
\includegraphics[scale=0.5]{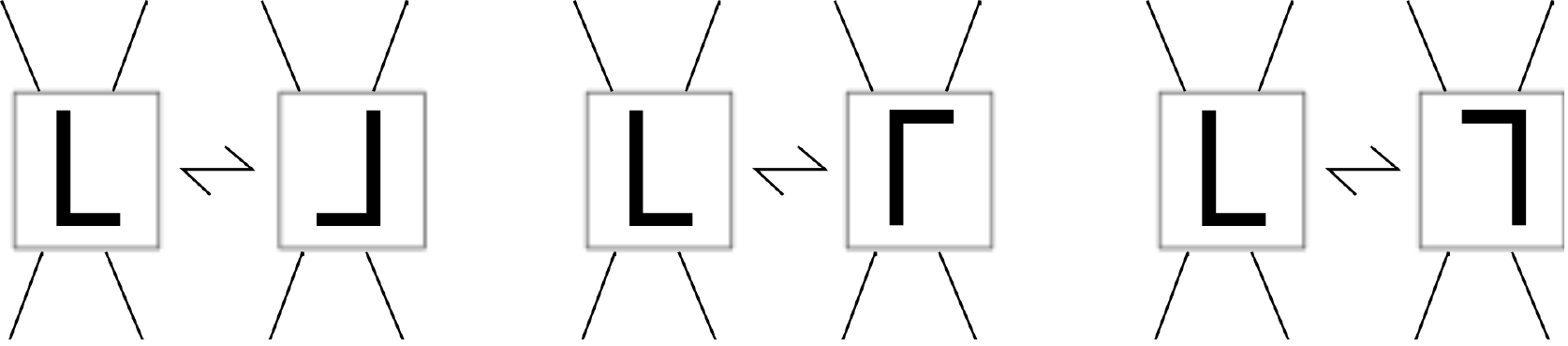}
\end{center}
\caption{Conway mutations on a tangle L within a knot diagram}
\end{figure}

\begin{definition}
A mutation is called \underline{\textit{positive}} if the orientation of the tangle matches after reglueing and is called \underline{\textit{negative}} if the orientation of the tangle needs to be reversed after reglueing.
\end{definition}

For more background on knots and mutations see \cite{Mut}.  After careful analysis of mutations up to isotopy, one sees that all positive Conway mutations done on a knot diagram can be realized by the 2 types of mutation in Figure 23.  We shall refer to them as \textit{positive rotation} and \textit{positive reflection}.\\

\begin{figure}[h]
\begin{center}
\includegraphics[scale=0.5]{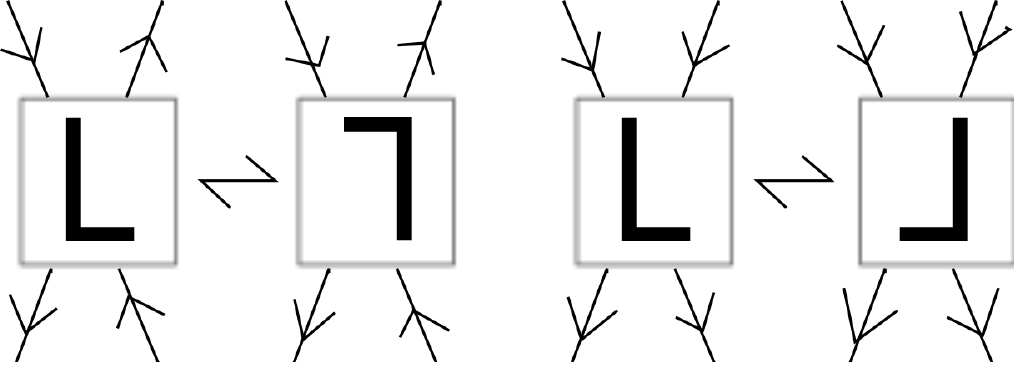}
\end{center}
\caption{Positive Rotation and Positive Reflection}
\end{figure}

We will illustrate the Affine Index Polynomial can detect mutations that are positive rotations with a family of examples.  We will then show that the polynomial is invariant under mutations that are postive reflections.  We leave the analysis of how the polynomial behaves under other types of mutation to another paper.\\

\subsection{Mutation by Positive Rotation}
We construct a family of examples that show the Affine Index Polynomial can detect mutation by positive rotation.  Consider the example in Figure 24 and Figure 25, where the mutation occurs in the dashed box.   $W_{K1}(t) = t + t^{-1} - 2$ and $W_{MK1}(t) = -t^4 + 3t - t^{-1} -1$.\\

\begin{figure}[htb]
\begin{minipage}[c]{0.38\textwidth}
\begin{tabular}{c|c|c|}
 & $W_+$ & $W_-$ \\\hline
A & 0 & 0\\\hline
B & +1 & -1 \\\hline
C1 & +1 & -1\\\hline
D & -1 & +1 \\\hline
E & -1 & +1\\\hline
\end{tabular}
\end{minipage}
\begin{minipage}{0.58\textwidth}
\includegraphics[scale = 0.4]{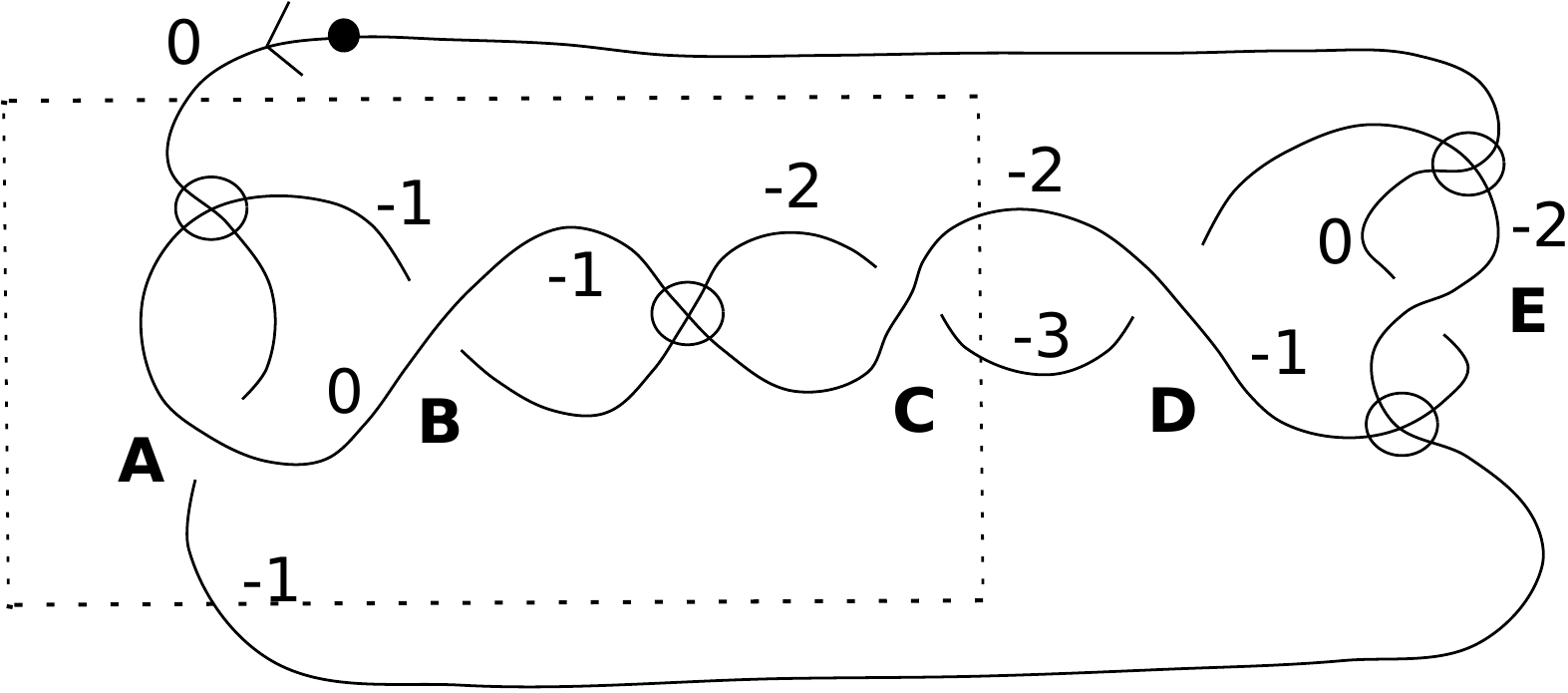}
\end{minipage}
\caption{Virtual Knot K1}
\end{figure}

\begin{figure}[htb]
\begin{minipage}[c]{0.38\textwidth}
\begin{tabular}{c|c|c|}
 & $W_+$ & $W_-$ \\\hline
A & -4 & +4 \\\hline
B & +1 & -1 \\\hline
C1 & +1 & -1 \\\hline
D & +1 & -1 \\\hline
E & +1 & -1 \\\hline
\end{tabular}
\end{minipage}
\begin{minipage}{0.58\textwidth}
\includegraphics[scale = 0.4]{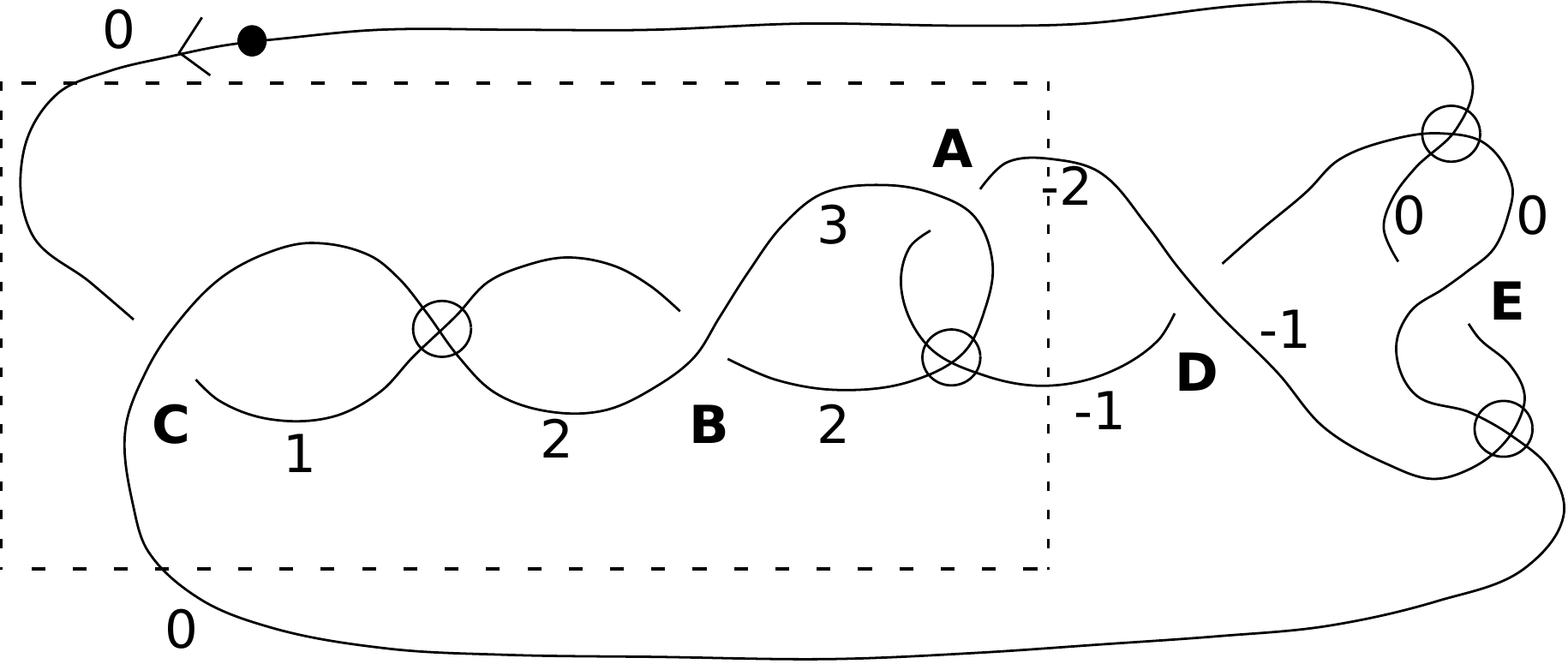}
\end{minipage}
\caption{Virtual Mutant Knot MK1}
\end{figure}

It is interesting to point out here that since the Wriggle Polynomial is equivalent to the Affine Index Polynomial, the Vassiliev Invariants coming from both polynomials are the same.  Recall that after substitution of $e^x$ the coefficients of $x^k$ in the wriggle polynomial are Vassiliev invariants of finite type $\displaystyle \lceil \frac{n}{2} \rceil$.  The n-th Vassiliev is given by the formula \cite{AIP}:\\
\[
v_n(K) = (1/n!) \sum_c sign(c) W(L_c)^n
\]
In classical knots, it has been shown that there are no Vassiliev Invariants of less than type 11 which distinguish mutant knots. We have just shown an example of a finite type 2 invariant that distinguishes the virtual mutant knots K1 and MK1.\\

In Figure 26 and Figure 27, we see K2 and its mutant, MK2 are also distinguished by this polynomial.  $W_{K2}(t) = t^{-3} + 3t$ and $W_{MK2}(t) = 2t + 2t^{-1}$.\\ 

\begin{figure}[htb]
\begin{minipage}[c]{0.38\textwidth}
\begin{tabular}{c|c|c|}
 & $W_+$ & $W_-$ \\\hline
A & -3 & +3\\\hline
B & +1 & -1 \\\hline
C1 & +1 & -1\\\hline
C2 & -1 & +1 \\\hline
D & +1 & -1 \\\hline
E & +1 & -1\\\hline
\end{tabular}
\end{minipage}
\begin{minipage}{0.58\textwidth}
\includegraphics[scale = 0.4]{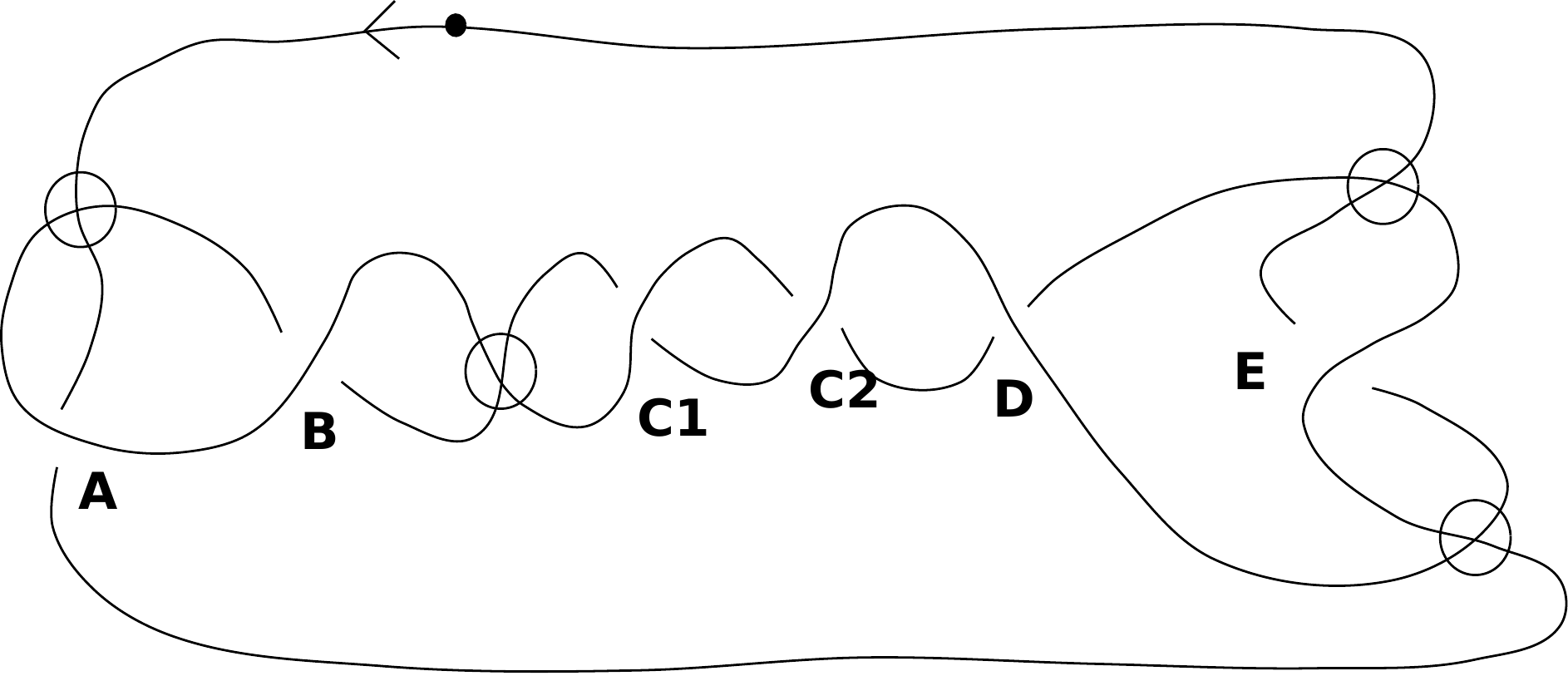}
\end{minipage}
\caption{Virtual Knot K2}
\end{figure}

\begin{figure}[htb]
\begin{minipage}[c]{0.38\textwidth}
\begin{tabular}{c|c|c|}
 & $W_+$ & $W_-$ \\\hline
A & +1 & -1  \\\hline
B & +1 & -1 \\\hline
C1 & +1 & -1 \\\hline
C2 & -1 & +1 \\\hline
D & -1 & +1 \\\hline
E & -1 & +1 \\\hline
\end{tabular}
\end{minipage}
\begin{minipage}{0.58\textwidth}
\includegraphics[scale = 0.4]{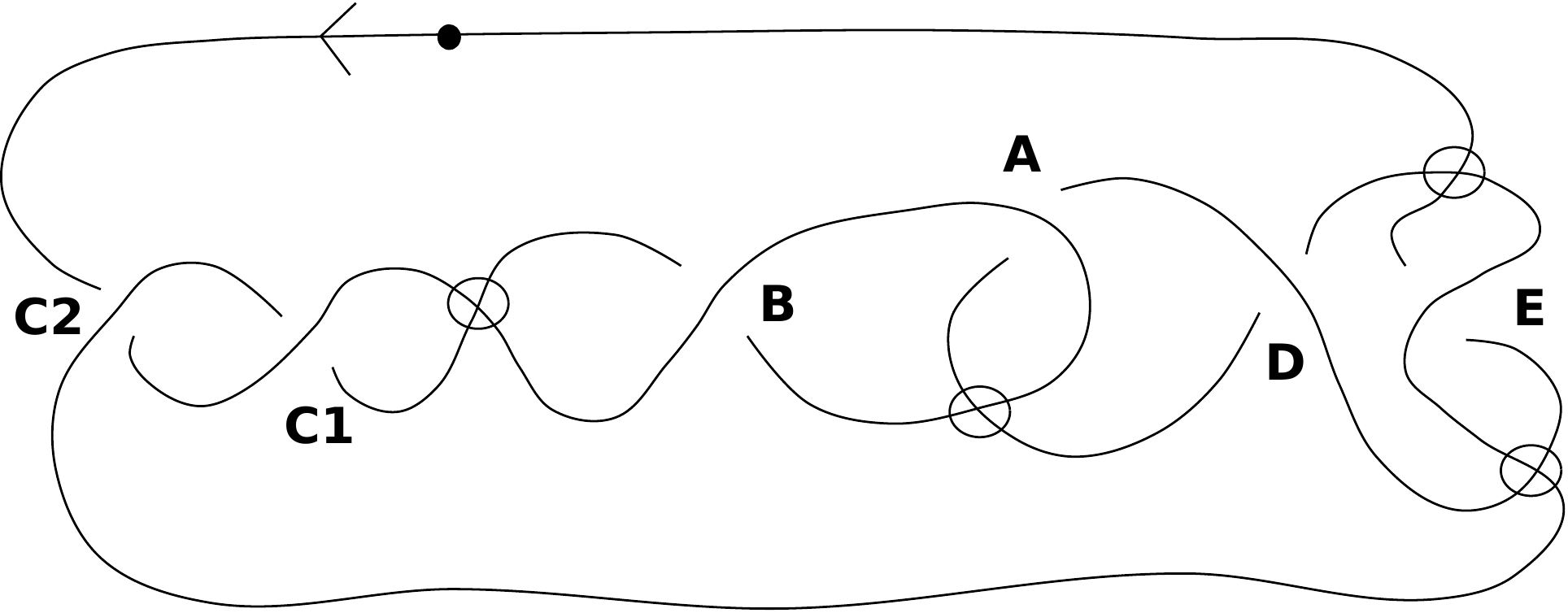}
\end{minipage}
\caption{Virtual Mutant Knot MK2}
\end{figure}

 We can create a family of examples, Kn, that can be distinguished from MKn by this polynomial by putting n-many twists at the crossing labeled C in K1 and C2 in K2.  Let us observe how the crossing replacement in Figure 28 for a knot diagram K affects $W_K(t)$.  Call the crossing on the left C.  C contributes the term $sign(C)t^{W(L_C)}$ to the polynomial.  The 3 twists on the right contribute $sign(C)[2t^{W(L_C)} + t^{-W(L_C)} - 2]$ to the polynomial without affecting the other terms.  By starting with knots K1/MK1 and K2/MK2, we can continue adding twists and recursively find the value of the polynomial for Kn/MKn.\\

\begin{figure}[h]
\begin{center}
\includegraphics[scale =0.5]{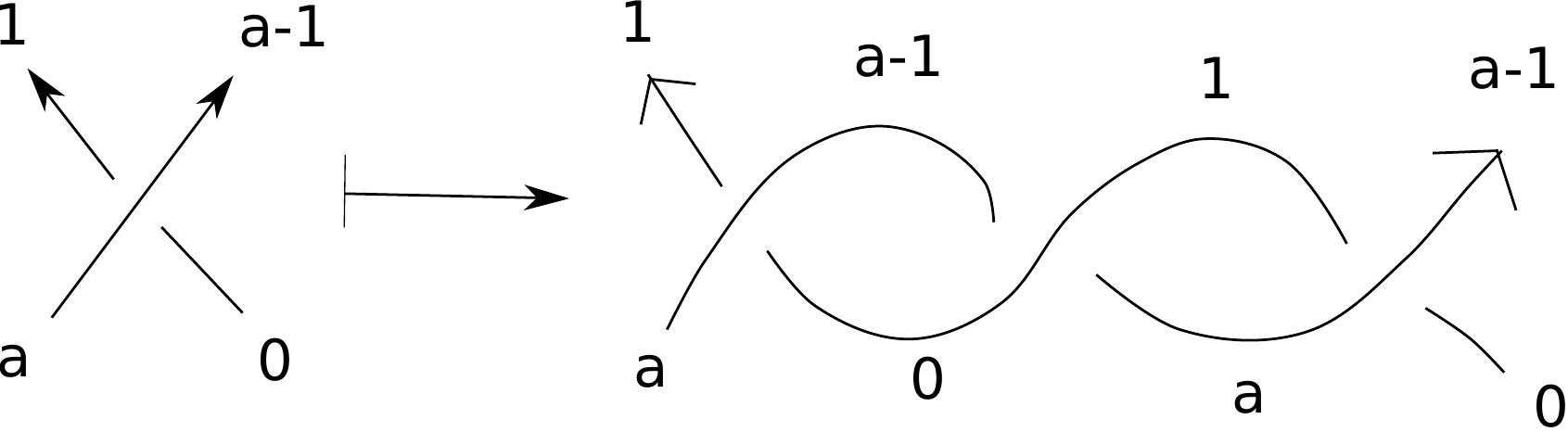}
\caption{A particular crossing replacement}
\end{center}
\end{figure}

\begin{center}
\begin{tabular}{c|c|c}
 & $W_{Kn}(t)$ & $W_{MKn}(t)$ \\\hline
n = 1 & $t + t^{-1} - 2$ & $-t^4 + 3t - t^{-1} - 1$ \\\hline
n = 3 & $2t + 2t^{-1} - 4$ & $-t^4 + 4t - 3$\\\hline
n = 5 & $3t + 3t^{-1} - 6$ & $-t^4 + 5t + t^{-1} - 5$\\\hline
n odd & $\displaystyle\frac{n+1}{2}t + \frac{n+1}{2}t^{-1} - (n+1)$ & $-t^4 + \displaystyle\frac{n+5}{2}t + \frac{n-3}{2}t^{-1} - n$ \\
\end{tabular}

\bigskip

\begin{tabular}{c|c|c}
 & $W_{Kn}(t)$ & $W_{MKn}(t)$ \\\hline
n = 2 & $t^{-3} + 3t$ & $2t + 2t^{-1}$ \\\hline
n = 4 & $t^{-3} + 4t + t^{-1} - 2$ & $ 3t + 3t^{-1} - 2$\\\hline
n = 6 & $t^{-3} + 5t + 2t^{-1} - 4$ & $4t + 4t^{-1} - 4$ \\\hline
n even & $t^{-3} + \displaystyle\frac{n+4}{2}t + \frac{n-2}{2}t^{-1} - (n-2)$ & $\displaystyle\frac{n+2}{2}t + \frac{n+2}{2}t^{-1} - (n-2)$ \\
\end{tabular}
\end{center}

Therefore, the family of mutants in Figure 29 is distinguished by the Wriggle Polynomial.\\

\begin{figure}[h]
\begin{minipage}[c]{0.5\textwidth}
\includegraphics[scale = 0.3]{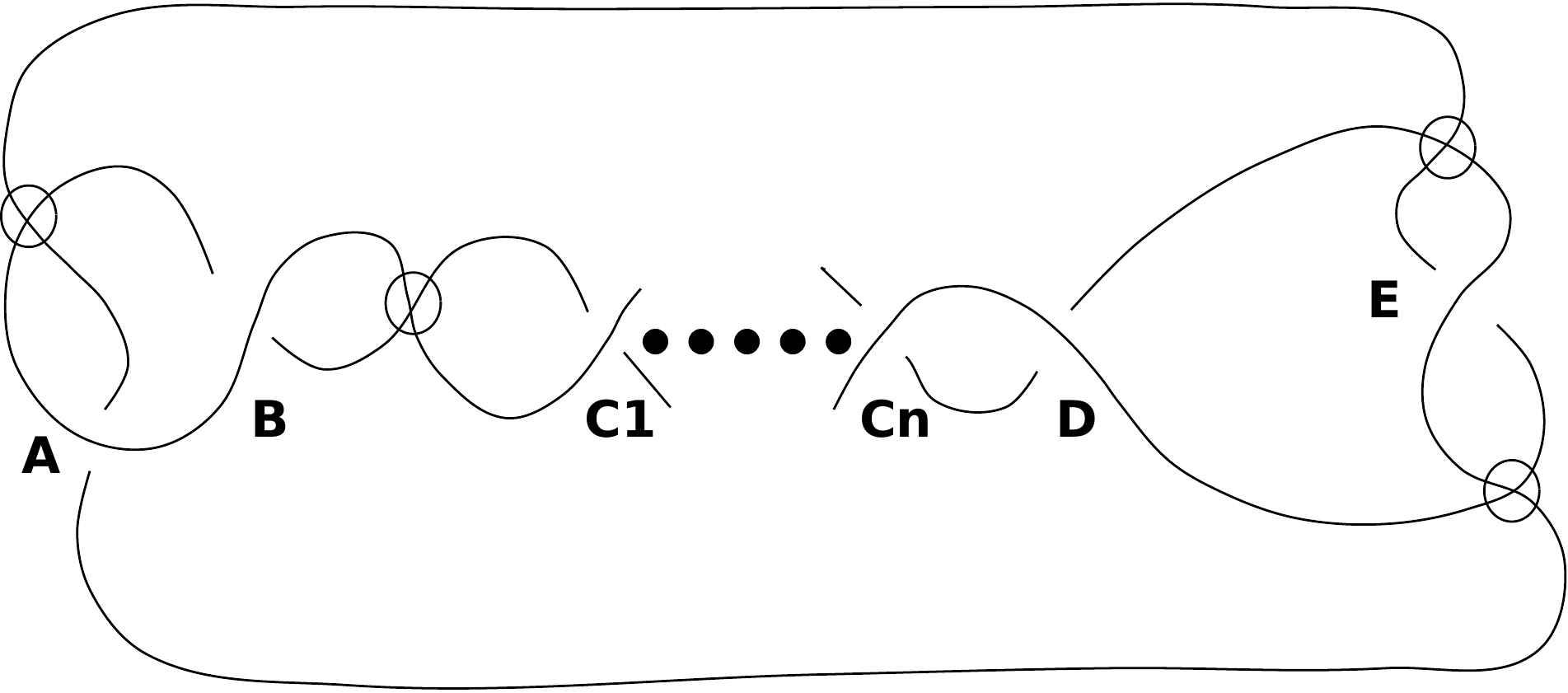}
\end{minipage}
\begin{minipage}{0.5\textwidth}
\includegraphics[scale = 0.3]{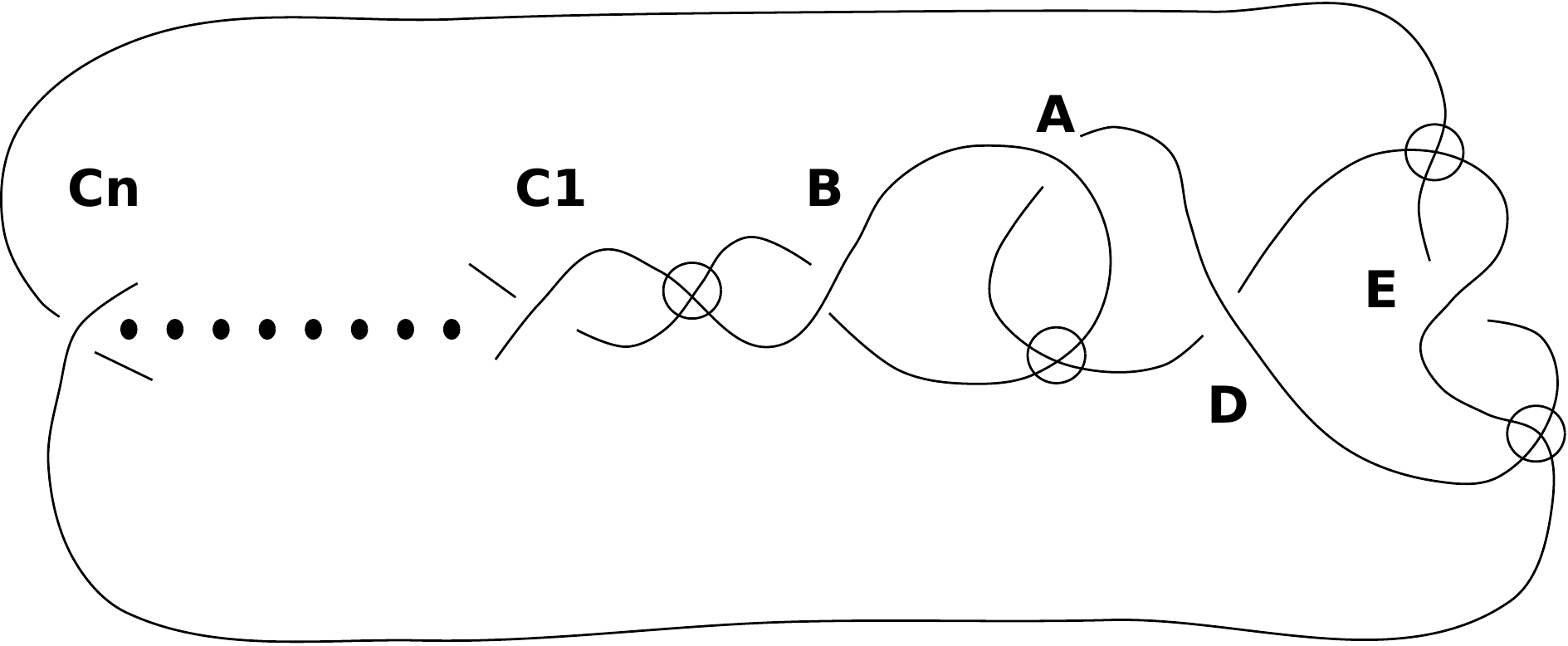}
\end{minipage}
\caption{The knots Kn  and MKn (before and after mutation)}
\end{figure}

We leave the classification of the conditions under which the Affine Index Polynomial detects mutation by positive rotation to another paper.

\subsection{Mutations and the Gauss code}

In the proof that the Affine Index Polynomial does not detect certain types of mutations we find ourselves analyzing the Gauss code of a knot diagram before and after a mutation occurs.  This requires us to translate the definition of the Affine Index Polynomial into a different format - the Gauss code.  Notice that given a signed Gauss code for a knot K, the letters represent crossings, and the spaces between the letters represent arcs.  Figure 30 shows how a labeling of arcs on a knot diagram can be incorporated in the Gauss code.  \\

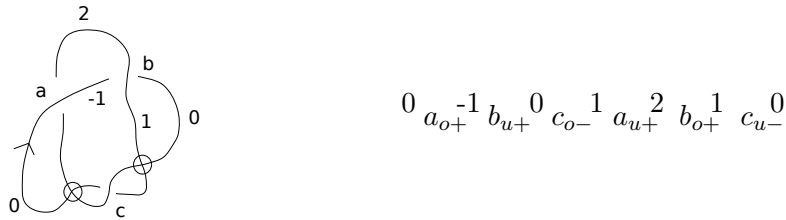
\begin{figure}[h]
\begin{center}
\begin{minipage}{0.3\textwidth}
\includegraphics[scale=0.4]{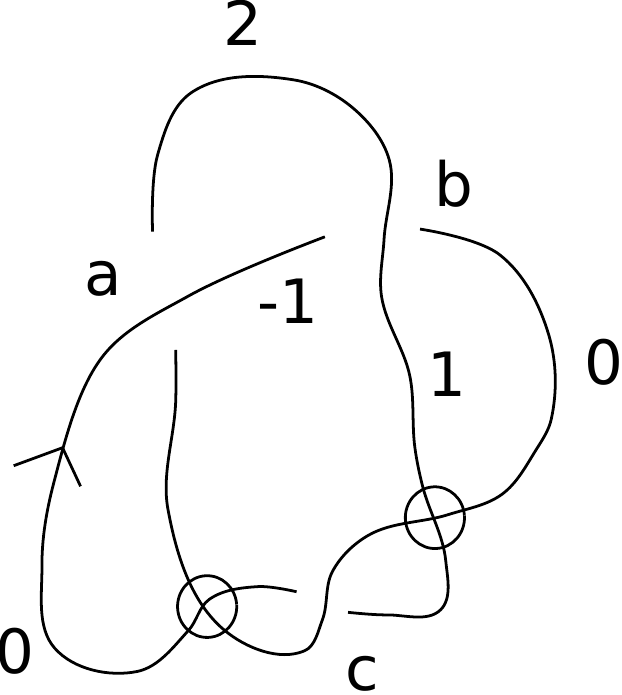}
\end{minipage}
\begin{minipage}{0.3\textwidth}
\begin{tikzpicture}  
  \node[anchor=base] at (1,2)
                {$a_{o+}  \ \ b_{u+} \ \ c_{o-} \ \ a_{u+} \ \ b_{o+} \ \ c_{u-} $};
  \node at (-1.6,2.3) {0};
  \node at (-0.8,2.3) {-1};
  \node at (0.1,2.3) {0};
  \node at (0.9,2.3) {1};
  \node at (1.7 , 2.3) {2};
  \node at (2.5 , 2.3) {1};
  \node at (3.3 , 2.3) {0};
  \end{tikzpicture}
\end{minipage}
\end{center}
\caption{Virtualized Trefoil with arcs labeled in Gauss code}
\end{figure}

Calculating the Affine Index Polynomial requires calculating a weight for each crossing.  In the Gauss code, the rule for calculating the weight at each crossing translates to taking ``outside differences" or ``inside differences" of arc labels, from the overcrossing to the undercrossing, as we shall explain below.  Figures 31 and 32 illustrate how the labeling rule and the weight calculation from the Affine Index Polynomial are interpreted on the Gauss code.\\

\begin{figure}[h]
\begin{center}
\begin{minipage}{0.4\textwidth}
\includegraphics[scale=0.5]{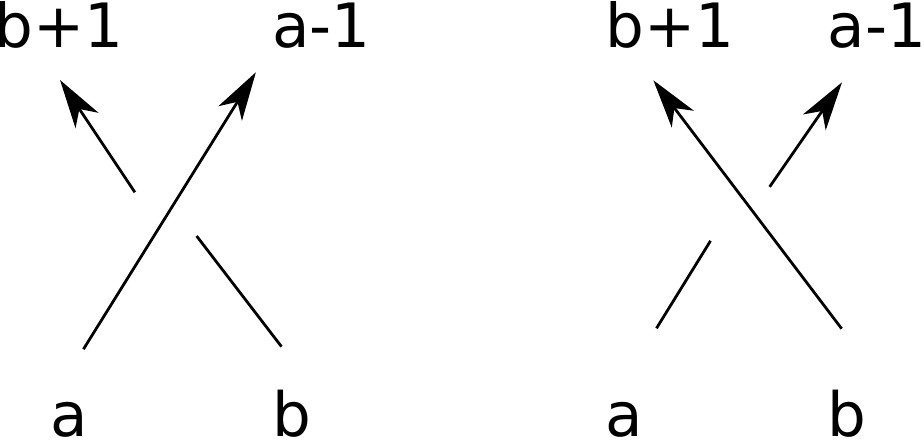}
\end{minipage}
\begin{minipage}{0.5\textwidth}
\begin{tabular}{|c|c|}
\hline
Traveling right of... & Arc label changes by ... \\\hline
o+ & -1 \\\hline
u+ & +1 \\\hline
o- & +1 \\\hline
u- & -1 \\\hline
\end{tabular}
\end{minipage}
\end{center}
\caption{Interpreting the labeling rule on the Gauss code}
\end{figure}

\begin{figure}[h]
\begin{minipage}{0.3 \textwidth}
\[
W_-(c) = b - (a - 1)
\]
\[
W_+(c) = a - (b+1)
\]
\end{minipage}
\begin{minipage}{0.2 \textwidth}
\[
... {}^b c_{o-} ... c_{u-}^{a-1}  ... 
\]
\[
... {}^a c_{o+} ... c_{u+}^{b+1} ...
\]
\end{minipage}
\begin{minipage}{0.3 \textwidth}
\[
W_{\pm}(c) = [(\text{\# before going over}) - (\text{\# after going under})]
\]
\[
\ \ \ \ \ \ \ \ \ \ \ \ \ \ = \text{Outside Differences} = \text{Inside Differences}
\]
\end{minipage}
\caption{Interpreting the weight calculation on the Gauss code}
\end{figure}

As an example, let us calculate the weight of each crossing in Figure 30 from the Gauss code.  Recall that we use the outside/inside differences rule going from the overcrossing to the undercrossing.  The weight of crossing a is (0 - 2) = -2 when calculated using outside differences, and (-1 - 1) = -2 when calculated using inside differences.  The weight of crossing b is (2 - 0) = 2 when calculated using outside differences and (1 - -1) = 2 when calculated using inside differences.  Notice that if we instead use the outside/inside differences rule from the undercrossing to the overcrossing, we get the same number with opposite sign.

\begin{definition}
A \underline{\textit{block}} of letters in a Gauss code is a connected subset of the Gauss code.  In Figure 31, bca is a block, while ac is not a block.
\end{definition}

\begin{definition}
Considering a particular block of letters in the Gauss code, a crossing in that block can be labeled as \underline{\textit{homebody}} or \underline{\textit{traveler}}.  A \underline{\textit{homebody crossing}} appears twice in the block of the Gauss code.  A \underline{\textit{traveler crossing}} appears once in the block of the Gauss code.
\end{definition}

\begin{definition}
A \underline{mutant block} in the Gauss code is a block of crossings affected/altered after a tangle in the knot undergoes a mutation.
\end{definition}

Conway mutations of a tangle in a knot diagram affect exactly 2 blocks of crossings in the Gauss code.  This is because when traversing the knot diagram, we enter (and leave) the tangle being mutated exactly twice.  Both times we enter and leave the tangle that we want to mutate, our travels in the tangle denote a block of letters in the Gauss code.

\begin{lemma}
A mutation by positive reflection changes the Gauss Code of the knot by:\\
1) Exchanging the 2 mutant blocks\\
2) Switching all over/under information associated to each occurance of a crossing.
\end{lemma}

\begin{proof} The proof follows after careful analysis of how mutation affects the Gauss code.
\end{proof}

To illustrate the lemma, consider the tangle in Figure 33.\\

\begin{figure}[h]
\begin{center}
\includegraphics[scale=0.5]{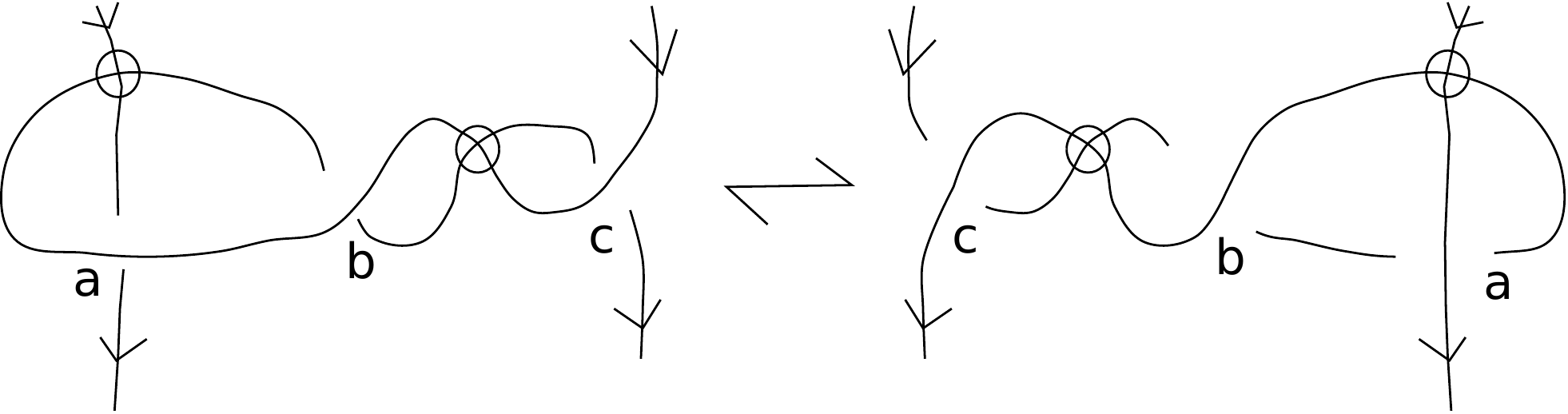}
\end{center}
\caption{An example of a tangle undergoing mutation by positive reflection}
\end{figure}

In Figure 33, let K be the name of the knot containing the tangle on the left, and MK be the name of the knot containing the tangle on the right (the tangle which has just undergone mutation).  Observe how the mutant blocks in the Gauss code have altered, illustrating the lemma.  These blocks have been underlined in the Gauss code for clarity.  The crossings outside the mutant blocks remain unchanged.
\[
\text{Gauss code for K} = \underline{a_{u+}} \dots \underline{c_{o+}b_{o+}a_{o+}b_{u+}c_{u+}} \dots
\]

\[
\text{Gauss code for MK} = \underline{c_{u+}b_{u+}a_{u+}b_{o+}c_{o+}} \dots \underline{a_{o+}} \dots
\]

This type of mutation exchanges the mutant blocks and switches the over/under designation associated with each instance of a crossing.

\subsection{Mutation by Positive Reflection}

\begin{theorem}
The Affine Index Polynomial does not detect mutation by positive reflection.
\end{theorem}
\begin{proof} 
We will show that for any crossing c in the Gauss code, the weight associated to that crossing, W(c), does not change after the Gauss code undergoes mutation by positive reflection.  By Lemma 4, since the sign of c does not change after this type of mutation, we conclude that the Affine Index Polynomial does not change.\\  

We show W(c) remains invariant under mutation by positive reflection by analyzing the \textit{net differences} in arc labels of blocks of crossings in the Gauss code.  Consider Figure 34 - the Gauss code of a knot before and after this type of mutation, with the arcs labeled with the rules from Figure 32.  In Figure 34, $m_i$ and $\hat{m}_j$ represent crossings in the mutant tangle and $x_i$ and $y_j$ represent crossings outside the mutant tangle.  The line segment $\Delta_i$ represents the net change in the arc labels between the points denoted by the line segment.\\

\begin{figure}[h]
\begin{center}
\begin{tikzpicture}  
\node[anchor=base] at (-5, 0.9)
{Knot before mutation};  
  \node[anchor=base] at (1,0.7)
                {$m_1 m_2 \ldots m_{k_1} x_1 \ldots x_{k_2} \hat{m}_1 \ldots \hat{m}_{k_3} y_1 \ldots y_{k_4} $};
  \draw[thick, rounded corners=4pt] (-2.5,1.1) -- (4.5,1.1);
  \draw[thick, rounded corners=4pt] (-2.5,0.9) -- (-2.5,1.1);
  \draw[thick, rounded corners=4pt] (-0.3,0.9) -- (-0.3,1.1);
  \draw[thick, rounded corners=4pt] (1.2,0.9) -- (1.2,1.1);
  \draw[thick, rounded corners=4pt] (2.8,0.9) -- (2.8,1.1);
  \draw[thick, rounded corners=4pt] (4.5,0.9) -- (4.5,1.1);
  \node at (-1.3,1.3) {$\Delta_1$};
  \node at (0.5,1.3) {$\Delta_2$};
  \node at (2,1.3) {$\Delta_3$};
  \node at (3.6,1.3) {$\Delta_4$};
\node[anchor=base] at (-5, -0.9)
{Knot after mutation}; 
 \node[anchor=base] at (1,-1)
                {$ \hat{m}_1 \ldots \hat{m}_{k_3}  x_1 \ldots x_{k_2}  m_1  \ldots m_{k_1}  y_1 \ldots y_{k_4} $};
  \draw[thick, rounded corners=4pt] (-2.3,-0.6) -- (4.3,-0.6);
  \draw[thick, rounded corners=4pt] (-2.3,-0.7) -- (-2.3,-0.6);
  \draw[thick, rounded corners=4pt] (-0.5,-0.7) -- (-0.5,-0.6);
  \draw[thick, rounded corners=4pt] (0.9,-0.7) -- (0.9,-0.6);
  \draw[thick, rounded corners=4pt] (2.7,-0.7) -- (2.7,-0.6);
  \draw[thick, rounded corners=4pt] (4.3,-0.7) -- (4.3,-0.6);
  \node at (-1.3,-0.3) {$\bar{\Delta}_3$};
  \node at (0.5,-0.3) {$\bar{\Delta}_2$};
  \node at (2,-0.3) {$\bar{\Delta}_1$};
  \node at (3.6,-0.3) {$\bar{\Delta}_4$};
  \end{tikzpicture}
\end{center}
\caption{Gauss code with net changes in arcs labels denoted}
\end{figure}
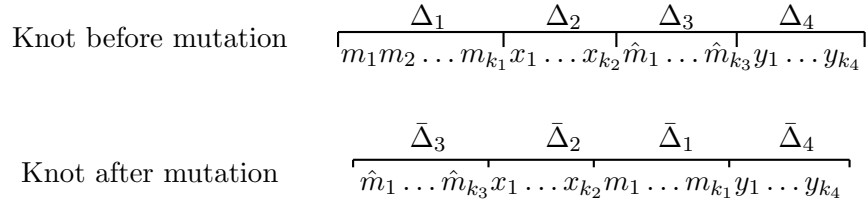

As an example, allow the ``knot before mutation" to be K1 (refer to Figure 24).  The net change of the left mutant block, $\Delta_1$, is -1 and the net change of the right mutant block, $\Delta_3$, is +1.

\begin{lemma} 
$\bar{\Delta}_1 = -\Delta_1$ , $\bar{\Delta}_3 = -\Delta_3$ , $\bar{\Delta}_2 = \Delta_2$ , $\bar{\Delta}_4 = \Delta_4$
\end{lemma}
\begin{proof}  Recall that this type of mutation exchanges the 2 mutant blocks, keeps the sign associated with each crossing the same, and switches the over/under information associated to each crossing.  Within the blocks denoted by $\bar{\Delta}_1$ and $\bar{\Delta}_3$, they differ from $\Delta_1$ and $\Delta_3$, respectively, only by the over/under information associated to each crossing being switched.  For a particular crossing, c, the information associated to c changes as follows after mutation:\\

\begin{center}
o+ : -1 $\longleftrightarrow$ u+ : +1\\
o - : +1 $\longleftrightarrow$ u - : -1\\
\end{center}

Each crossing switches its contribution to the arc labels from +1 to -1 (or vice versa).  When calculating the contribution to the net change in arc labels for the block we get $\bar{\Delta}_1 = -\Delta_1$ and $\bar{\Delta}_3 = -\Delta_3$.  Since the blocks sectioned off by $\Delta_2$ and $\Delta_4$ remain unchanged after a mutation, it is clear the the net difference of the arc labels from those blocks remains unchanged.
\end{proof}

\begin{lemma}
$\Delta_1 = -\Delta_3$
\end{lemma}

\begin{proof}
Recall that crossings in $\Delta_i$ can be categorized as homebody or traveler.  WLOG assume the first occurence of crossing c in the Gauss code is in the left mutant bock, and c is a positive overcrossing ($c_{o+}$).  If c is a homebody crossing, then both occurences of c, $c_{o+}$ and $c_{u+}$, occur in the left mutant block.  $c_{o+}$ contributes -1 to $\Delta_1$ and $c_{u+}$ contributes +1 to $\Delta_1$.  Therefore if c is a homebody crossing, it has a net contribution of 0 to $\Delta_1$.  If c is a traveler crossing, then $c_{o+}$ contributes -1 to $\Delta_1$, and $c_{u+}$ contributes +1 to $\Delta_3$.  Therefore $\Delta_1$ = -$\Delta_3$
\end{proof}

By applying the same argument as above we obtain the following corollary.

\begin{corollary}
$\Delta_2$ = -$\Delta_4$
\end{corollary}

By applying the lemmas above, we can consider the Gauss code in Figure 35 during the proof of the following lemmas.\\

\begin{figure}[h]
\begin{center}
\begin{tikzpicture}
\node[anchor=base] at (-5, 0.9)
{Knot before mutation};  
  \node[anchor=base] at (1,0.7)
                {$m_1 m_2 \ldots m_{k_1} x_1 \ldots x_{k_2} \hat{m}_1 \ldots \hat{m}_{k_3} y_1 \ldots y_{k_4} $};
  \draw[thick, rounded corners=4pt] (-2.5,1.1) -- (4.5,1.1);
  \draw[thick, rounded corners=4pt] (-2.5,0.9) -- (-2.5,1.1);
  \draw[thick, rounded corners=4pt] (-0.3,0.9) -- (-0.3,1.1);
  \draw[thick, rounded corners=4pt] (1.2,0.9) -- (1.2,1.1);
  \draw[thick, rounded corners=4pt] (2.8,0.9) -- (2.8,1.1);
  \draw[thick, rounded corners=4pt] (4.5,0.9) -- (4.5,1.1);
  \node at (-1.3,1.3) {$\Delta_1$};
  \node at (0.5,1.3) {$\Delta_2$};
  \node at (2,1.3) {$-\Delta_1$};
  \node at (3.6,1.3) {$-\Delta_2$};
\node[anchor=base] at (-5, -0.9)
{Knot after mutation};  
 \node[anchor=base] at (1,-1)
                {$ \hat{m}_1 \ldots \hat{m}_{k_3}  x_1 \ldots x_{k_2}  m_1  \ldots m_{k_1}  y_1 \ldots y_{k_4} $};
  \draw[thick, rounded corners=4pt] (-2.3,-0.6) -- (4.3,-0.6);
  \draw[thick, rounded corners=4pt] (-2.3,-0.7) -- (-2.3,-0.6);
  \draw[thick, rounded corners=4pt] (-0.5,-0.7) -- (-0.5,-0.6);
  \draw[thick, rounded corners=4pt] (0.9,-0.7) -- (0.9,-0.6);
  \draw[thick, rounded corners=4pt] (2.7,-0.7) -- (2.7,-0.6);
  \draw[thick, rounded corners=4pt] (4.3,-0.7) -- (4.3,-0.6);
  \node at (-1.3,-0.3) {$\Delta_1$};
  \node at (0.5,-0.3) {$\Delta_2$};
  \node at (2,-0.3) {$-\Delta_1$};
  \node at (3.6,-0.3) {$-\Delta_2$};
  \end{tikzpicture}
\end{center}
\caption{Gauss code with net changes in arc labels denoted}
\end{figure}

\begin{lemma}
If c is a crossing outside the mutant blocks, W(c) remains unchanged after mutation.
\end{lemma}

\begin{proof}
If c is a crossing outside the mutant blocks, c is contributing to $\Delta_2$ or $-\Delta_2$.  W(c) remains unchanged as a result of mutation leaving the net change of the left mutant block (denoted by $\Delta_1$) and the right mutant block (denoted by $-\Delta_1$)  the same.
\end{proof}

\begin{lemma}
If c is a homebody crossing in a mutant block, W(c) remains unchanged after mutation.
\end{lemma}

\begin{proof}
WLOG we may assume that both occurences of c occur in the left mutant block in the Gauss code before mutation, thus appearing in the right mutant block after mutation.  WLOG we may also assume that we first encounter c as a positive overcrossing ($c_{o+}$).  Consider the net changes to the arc labels after mutation drawn in Figure 36.\\

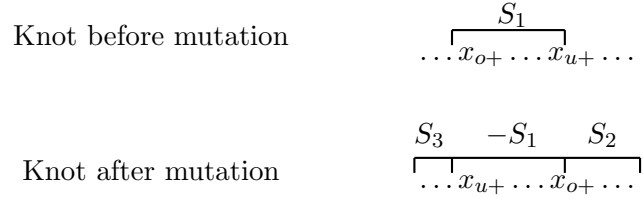
\begin{figure}[h]
\begin{center}
\begin{tikzpicture}
\node[anchor=base] at (-4, 0.9)
{Knot before mutation};  
  \node[anchor=base] at (1,0.7)
                {$\ldots x_{o+} \ldots x_{u+} \ldots $};
  \draw[thick, rounded corners=4pt] (0,1.1) -- (1.5,1.1);
  \draw[thick, rounded corners=4pt] (0,0.9) -- (0,1.1);
  \draw[thick, rounded corners=4pt] (1.5,0.9) -- (1.5,1.1);
  \node at (0.8,1.3) {$S_1$};
  \node[anchor=base] at (-4, -0.9)
{Knot after mutation};  
 \node[anchor=base] at (1,-1)
                {$\ldots x_{u+} \ldots x_{o+} \ldots $};
  \draw[thick, rounded corners=4pt] (-0.5,-0.6) -- (2.5,-0.6);
  \draw[thick, rounded corners=4pt] (0,-0.8) -- (0,-0.6);
  \draw[thick, rounded corners=4pt] (1.5,-0.8) -- (1.5,-0.6);
  \draw[thick, rounded corners=4pt] (-0.5,-0.8) -- (-0.5,-0.6);
  \draw[thick, rounded corners=4pt] (2.5,-0.8) -- (2.5,-0.6);
  \node at (0.8,-0.3) {$-S_1$};
  \node at (2, -0.3) {$S_2$};
  \node at (-0.3, -0.3) {$S_3$};

\end{tikzpicture}
\end{center}
\caption{Gauss Code for Homebody Crossing in a Mutant Block}
\end{figure}

Recall that W(c) is calculated as the ``outside differences" or ``inside differences" from the overcrossing to the undercrossing.  Before mutation, we calculate W(c) = $S_1$ using outside differences.   After mutation, W(c) = $S_2 + S_3$ if calculated as the ``inside differences" from overcrossing to undercrossing.  WLOG we may assume the arc before crossing $x_{u+}$ is labeled 0.\\
\[
-S_1 + S_2 + S_3 = 0
\]
\[
S_2 + S_3 = S_1
\]
\[
W(c) = S_1
\]
Alternatively, we could calculate W(c) in the reverse direction, by reading undercrossing to overcrossing.  Then, using outside differences we see W(c) = - (-$S_1$) = $S_1$.
\end{proof}

\begin{lemma}
If c is a traveler crossing in a mutant block, W(c) remains unchanged after mutation.
\end{lemma}

\begin{proof}
WLOG we may assume c is positive and we first encounter c as an overcrossing ($c_{o+}$) in the left mutant block.  Thus the first instance of c in the left mutant block, after mutation, is still an overcrossing.  Consider the net changes to the arc labels in Figure 37.\\

\begin{figure}[h]
\begin{center}
\begin{tikzpicture}
\node[anchor=base] at (-6, 0.7)
{Knot before mutation};  
  \node[anchor=base] at (1,0.7)
                {$m_1 \ldots c_{o+} \ldots m_{k_1} x_1 \ldots x_{k_2} \hat{m}_1 \ldots c_{u+} \ldots \hat{m}_{k_3} y_1 \ldots y_{k_4} $};
  \draw[thick, rounded corners=4pt] (-3.5,2) -- (5.5,2);
\draw[thick, rounded corners=4pt] (-3.5,1.3) -- (-0.6,1.3);
\draw[thick, rounded corners=4pt] (0.9,1.3) -- (3.7,1.3);
  \draw[thick, rounded corners=4pt] (-2.4,0.9) -- (-2.4,1.3);
  \draw[thick, rounded corners=4pt] (2.3,0.9) -- (2.3,1.3);
  \draw[thick, rounded corners=4pt] (-3.5,0.9) -- (-3.5,2);
  \draw[thick, rounded corners=4pt] (-0.6,0.9) -- (-0.6,2);
  \draw[thick, rounded corners=4pt] (0.9,0.9) -- (0.9,2);
  \draw[thick, rounded corners=4pt] (3.7,0.9) -- (3.7,2);
  \draw[thick, rounded corners=4pt] (5.5,0.9) -- (5.5,2);
  \node at (-2,2.3) {$\Delta_1$};
  \node at (0.3,2.3) {$\Delta_2$};
  \node at (2.4,2.3) {$-\Delta_1$};
  \node at (4.6,2.3) {$-\Delta_2$};
  \node at (-3,1.6) {$S_{11}$};
  \node at (-1.5,1.6) {$S_{12}$};
  \node at (1.6,1.6) {$S_{21}$};
  \node at (3,1.6) {$S_{22}$};
\node[anchor=base] at (-6, -2)
{Knot after mutation};  
 \node[anchor=base] at (1,-1.9)
                {$ \hat{m}_1 \ldots c_{o+} \ldots \hat{m}_{k_3}  x_1 \ldots x_{k_2}  m_1  \ldots c_{u+} \ldots m_{k_1}  y_1 \ldots y_{k_4} $};
  \draw[thick, rounded corners=4pt] (-3.5,-.6) -- (5.5,-.6);
\draw[thick, rounded corners=4pt] (-3.5,-1.3) -- (-0.6,-1.3);
\draw[thick, rounded corners=4pt] (0.9,-1.3) -- (3.7,-1.3);
  \draw[thick, rounded corners=4pt] (-2,-1.7) -- (-2,-1.3);
  \draw[thick, rounded corners=4pt] (2,-1.7) -- (2,-1.3);
  \draw[thick, rounded corners=4pt] (-3.5,-1.8) -- (-3.5,-.6);
  \draw[thick, rounded corners=4pt] (-0.6,-1.8) -- (-0.6,-0.6);
  \draw[thick, rounded corners=4pt] (0.9,-1.8) -- (0.9,-0.6);
  \draw[thick, rounded corners=4pt] (3.7,-1.8) -- (3.7,-0.6);
  \draw[thick, rounded corners=4pt] (5.5,-1.8) -- (5.5,-0.6);
  \node at (-2,-0.3) {$\Delta_1$};
  \node at (0.3,-0.3) {$\Delta_2$};
  \node at (2.4,-0.3) {$-\Delta_1$};
  \node at (4.6,-0.3) {$-\Delta_2$};
  \node at (-3,-1) {$-S_{21}$};
  \node at (-1.5,-1) {$-S_{22}$};
  \node at (1.6,-1) {$-S_{11}$};
  \node at (3,-1) {$-S_{12}$};
  \end{tikzpicture}
\end{center}
\caption{Gauss Code with net changes in Cheng Weights labeled}
\end{figure}
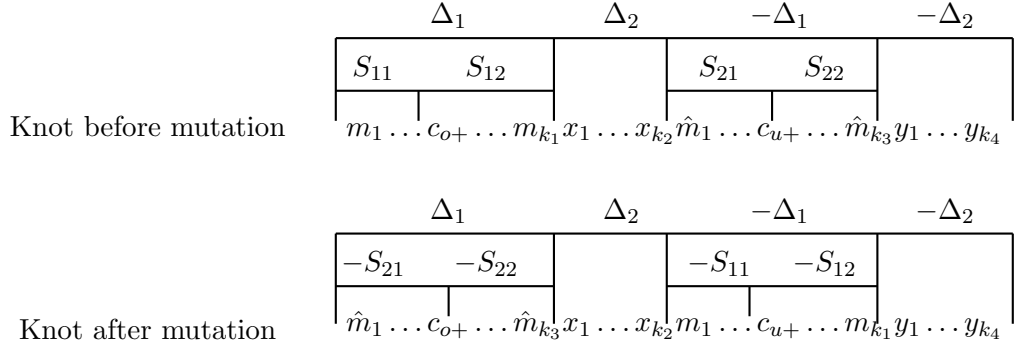
Before mutation, W(c) = 0 - $(S_{12} + \Delta_2 + S_{21}) = -S_{12} - \Delta_2 - S_{21}$ when calculated with outside differences from overcrossing to undercrossing.\\
After mutation, when reading left to right and calculating inside differences for W(c) we obtain
\[
W(c) = 0 - (-S_{22} + \Delta_2 + S_{11})
\] 
\[= +S_{22} - \Delta_2 - S_{11}
\]
With this calculation, it is not immediately clear that W(c) would remain unchanged after mutation.  Recall that alternatively we could calculate outside differences for W(c) in the reverse direction (by reading undercrossing to overcrossing) and take the opposite sign of the calculation.
\[
W(c) = - (0 - (-S_{12} - \Delta_2 - S_{21}))
\]
\[
= -(S_{12} + \Delta_2 + S_{21})
\]
\[
= -S_{12} - \Delta_2 - S_{21}
\]
\end{proof}

Thus for each possible category of crossing, we have shown that after mutation by positive reflection W(c) remains unchanged for all crossings c.  Therefore the Affine Index Polynomial remains unchanged and this type of mutation is not detected.
\end{proof}


\begin{thebibliography}{AF}

\bibitem{CCGen1} 
C. Balm, S. Friedl, E. Kalfagianni, M. Powell, Cosmetic Crossings and Seifert Matrices, {\em Communications in Analysis and Geometry} (to appear), arXiv:1108.3102

\bibitem{CS1}
J.S.Carter, S. Kamada and M. Saito, Stable equivalence of knots on surfaces and virtual  knot
cobordisms, in ``Knots 2000 Korea, Vol. 1 (Yongpyong)", {\em JKTR} {\bf 11}, No. 3 (2002), 311--320.

\bibitem{OW}
Zhiyun Cheng, A polynomial invariant of virtual knots. arXiv:1202.3850

\bibitem{DK}
H. Dye and L. H. Kauffman,  Minimal surface representations of virtual knots and links. 
{\it Algebr. Geom. Topol.}  {\bf 5} (2005), 509--535. 

\bibitem{FRR}
R. Fenn, R. Rimanyi, C. Rourke, The braid permutation group, {\em Topology}  {\bf 36} (1997), 123--135.

\bibitem{Fox} 
R. Fox, A quick trip through knot theory, in ``Topology of Three-Manifolds" ed. by M.K.Fort, Prentice-Hall Pub. (1962).

\bibitem{GPV}
 M. Goussarov, M.Polyak and O. Viro, Finite type invariants of classical and virtual knots,
{\em Topology}  {\bf 39} (2000),  1045--1068.

\bibitem{Hen}
A. Henrich, A Sequence of Degree One Vassiliev Invariants for Virtual Knots, {\em Journal of Knot
Theory and its Ramiﬁcations} {\bf 19, 4} (2010), pp. 461--487. arXiv:0803.0754

\bibitem{HR}
D. Hrencecin, ``On Filamentations and Virtual Knot Invariants" Ph.D Thesis, Unviversity
of Illinois at Chicago (2001).

\bibitem{HRK}
 D. Hrencecin and L. H. Kauffman, ``On Filamentations and Virtual Knots",  {\em Topology
and Its Applications}  {\bf 134} (2003), 23--52.

\bibitem{KADOKAMI}
T. Kadokami, Detecting non-triviality of virtual links,{\it   JKTR} {\bf 6}, No. 2  (2003), 781--803.

\bibitem{Effie} 
Efstratia Kalfagianni, Cosmetic crossing changes of fibered knots, {\em J. Reine Angew. Math (Crelle's Journal)}, DOI: 10.1515/CRELLE.2011.148, arXiv:math/0610440

\bibitem{Kamada}
 S. Kamada, Braid presentation of virtual knots and welded knots, preprint March 2000.

\bibitem{VKT}
L. H. Kauffman, Virtual Knot Theory , {\em European J. Comb.}  {\bf 20} (1999), 663--690.

\bibitem{SVKT} 
L. H. Kauffman, A Survey of Virtual Knot Theory, {\em Proceedings of Knots  in
Hellas~'98}, World Sci. 2000, 143--202.

\bibitem{DVK} 
L. H. Kauffman, Detecting Virtual Knots, {\em Atti. Sem. Mat. Fis. Univ. Modena
Supplemento al  Vol. IL } (2001),  241--282.

\bibitem{KL3} 
L. H. Kauffman, S. Lambropoulou, Virtual Braids, {\em Fund. Math.} {\bf 184} (2004), 159--186.

\bibitem{KL4}  
L. H. Kauffman, S. Lambropoulou, Virtual braids and the L-move, {\em JKTR} {\bf 15}, No. 6 (2006), 773--811.

\bibitem{KP} 
L. H. Kauffman,``Knots and Physics",  World Sci.  (1991), Second Edition 1994 (723 pages), Third Edition 2001.

\bibitem{SL} 
Louis H. Kauffman, math.GT/0405049, A self-linking invariant of virtual
knots. {\it Fund. Math.} 184 (2004), 135--158.

\bibitem{IVKT} 
Louis H. Kauffman, Introduction to Virtual Knot Theory
(to appear in Special Issue of JKTR on Virtual Knot Theory, November
2012), arXiv:1101.0665.

\bibitem{AIP} 
Louis H. Kauffman, The Affine Index Polynonmial for Virtual Knots and Links, arXiv:1211.1601

\bibitem{Kirby}
R. Kirby, Problems in low-dimensional topology. {\em AMS/IP Stud. Adv. Math.} (1997) 35-–473.  

\bibitem{Mut} 
Paul Kirk and Charles Livingston, Concordance and mutation, {\em Geom. Topol.} {\bf 5} (2001), 831–883

\bibitem{KiSa} 
T. Kishino and S. Satoh, A note on non-classical virtual knots, {\em JKTR} {\bf 13}, No. 7  (2004), 845--856.

\bibitem{KUP}
G. Kuperberg, What is a virtual link?, {\em Algebraic and Geometric Topology}  {\bf 3}
2003, 587--591.

\bibitem{Satoh} 
S. Satoh,  Virtual knot presentation of
ribbon torus-knots, {\em JKTR}  {\bf 9} (2000), No.4,  531--542.

\bibitem{TURAEV}
V. Turaev, Virtual strings and their cobordisms,
arXiv.math.GT/0311185.

\end{thebibliography}
\end{document}